\documentclass[12pt,oneside]{amsart}
\usepackage{amsfonts}
\usepackage{amsmath}
\usepackage{amssymb}
\usepackage{mathrsfs}
\usepackage{geometry}
\usepackage{graphicx}
\usepackage{subfigure}
\usepackage{hyperref}
\usepackage{microtype}
\usepackage{enumitem}
\usepackage{geometry}
\usepackage{tikz}
\usepackage{tikz-cd}
\usepackage{color}
\numberwithin{equation}{section}
\allowdisplaybreaks[1]
\usepackage{titletoc}
\usepackage{amsthm}
\usepackage{cite}
 \usepackage{float}

\newcommand{\la}{\lambda}

\newcommand{\ve}{\varepsilon}

\newcommand{\ka}{\kappa}

\newcommand{\R}{\mathbb{R}}

\newcommand{\n}[1]{\Vert #1\Vert }
\newcommand{\bn}[1]{\big \Vert #1 \big \Vert }
\newcommand{\bbn}[1]{\Big\Vert #1 \Big \Vert }
\newcommand{\lr}[1]{\left\{ #1 \right\} }
\newcommand{\lrc}[1]{\left[ #1 \right] }
\newcommand{\lrs}[1]{\left( #1 \right) }

\newcommand{\lra}[1]{\langle #1 \rangle }

\newcommand{\abs}[1]{|#1| }

\newcommand{\wt}[1]{\widetilde{#1}}

\newcommand{\pa}{\partial}

\newcommand{\ol}{\overline}

\begin{document}

\newtheorem{theorem}{Theorem}[section]
\newtheorem{lemma}[theorem]{Lemma}

\theoremstyle{definition}
\newtheorem{definition}[theorem]{Definition}
\newtheorem{example}[theorem]{Example}
\newtheorem{remark}[theorem]{Remark}

\numberwithin{equation}{section}

\newtheorem{proposition}[theorem]{Proposition}
\newtheorem{corollary}[theorem]{Corollary}
\newtheorem{goal}[theorem]{Goal}

\title[Stability threshold for the hard-sphere Boltzmann equation]{Sharp global stability for the hard-sphere Boltzmann equation near equilibrium}

\author[X. Chen]{Xuwen Chen}
\address{Department of Mathematics, University of Rochester, Rochester, NY 14627, USA}
\email{xuwenmath@gmail.com}

\author[S. Shen]{Shunlin Shen}
\address{School of Mathematical Sciences, University of Science and Technology of China, Hefei, 230026, China}
\email{slshen@ustc.edu.cn}

\author[Z. Zhang]{Zhifei Zhang}
\address{School of Mathematical Sciences, Peking University, Beijing, 100871, China}

\email{zfzhang@math.pku.edu.cn}


\begin{abstract}
We study the hard-sphere Boltzmann equation near a global Maxwellian and
have found the bifurcation between the stability and instability in $H_{x}^{s}$ Sobolev space.
We prove a strong-weak type instability result in the $s<1$ regularity regime, and establish the global stability under the regularity $s>1$ via the interplay between the dissipation and dispersive effects.
 In this intricate competition of dispersion, dissipation, large-velocity growth, and the asymmetric nature of the gain and loss terms, we prove that the instability is jointly generated by the dispersive effect and the loss term structure. To reach the sharp regularity for the global stability, we incorporate the dispersive
effect into the dissipative energy framework for the analysis near equilibrium. This hence concludes sharp threshold and provides a complete characterization of global well-posedness near equilibrium.
 \end{abstract}
\maketitle
\tableofcontents

\section{Introduction}
We consider the hard-sphere Boltzmann equation
\begin{equation}\label{equ:Boltzmann}
\left\{
\begin{aligned}
\lrs{\pa_{t}+v\cdot\nabla_{x}} F(t,x,v)&=Q(F,F),\\
F(0,x,v)&=F^{\mathrm{in}}(x,v),
\end{aligned}
\right.
\end{equation}
where $F(t,x,v)$ denotes the distribution function of particles at time $t\geq 0$, position $x\in \mathbb{R}^{3}$, and velocity $v\in \mathbb{R}^{3}$. The collision operator $Q$ admits the decomposition
\begin{align*}
Q(F,G)&=Q^{+}(F,G)-Q^{-}(F,G),
\end{align*}
with the gain term and the loss term given respectively by
\begin{align*}
\text{gain term:}\quad Q^{+}(F,G)&=\int_{\mathbb{R}^{3}}\int_{\mathbb{S}^{2}} F(v^{*})G(u^{*}) |(u-v)\cdot \omega|\,dud\omega,\\
\text{loss term:}\quad Q^{-}(F,G)&=F(v)\int_{\mathbb{R}^{3}}\int_{\mathbb{S}^{2}} G(u) |(u-v)\cdot \omega| \,du d\omega.
\end{align*}
Here, the post-collision velocities $(v^*, u^*)$ are related to the pre-collision velocities $(v, u)$ via
\begin{align*}
u^{*}=u+\lrc{(v-u)\cdot \omega} \omega,\quad v^{*}=v-\lrc{(v-u)\cdot \omega} \omega.
\end{align*}

The hard-sphere Boltzmann equation has profound historical and theoretical significance. Originally formulated by Maxwell \cite{Max67} and Boltzmann \cite{Bol68}, it serves as a foundational model for describing non-equilibrium statistical mechanics. It occupies a key position in the mathematical theory of kinetic gases and provides a bridge between the deterministic description of classical Newtonian particle systems and the macroscopic fluid equations. The equation remains one of the few settings in which a rigorous derivation from classical particle dynamics has been established. For a comprehensive and detailed discussion, see, for example, \cite{BGSS23,DHM25,GST13,Lan75,Vil02} and the references therein.

Owing to its fundamental importance, the Cauchy problem for the hard-sphere Boltzmann equation has attracted extensive attention, and substantial progress has been made in its well-posedness theory. See, for example,
\cite{AMUXY13,Ars11,CDP19,DP89,DS18,HJKL24,KS78,IS84,Uka86}. Regarding the global existence theory, the pioneering work of DiPerna and Lions \cite{DP89} established the theory of global renormalized weak solutions with
large initial data.

A parallel and equally important line of research concerns stability near equilibrium. A Maxwellian distribution describes thermal equilibrium of a dilute gas and arises naturally from the conservation laws and entropy structure of the Boltzmann equation. It further serves as the key bridge between collisional relaxation and macroscopic fluid motion through the hydrodynamic limit, which yields fluid systems such as the Euler and Navier-Stokes equations.
The stability of Maxwellian equilibrium is consequently a fundamental problem in kinetic theory.
In this direction, substantial results have been obtained in the analysis of stability near a global Maxwellian, such as \cite{AEP87,CIP94,Duan08,DHWY17,DLX16,Guo03,Guo03vlasov,Guo04,Guo10,LYY04,LY04,LY11,NI77,Uka74}.

In contrast to the richness of these stability theories, however, the current understanding of instability near equilibrium for the Boltzmann equation remains quite limited, and the problem of characterizing the
bifurcation between stability and instability has so far remained completely
open. This naturally leads to a more basic question, namely where the threshold between
stability and instability near a Maxwellian lies and which structural features
of the Boltzmann equation determine that threshold.

An important clue comes from recent progress on the bifurcation between well-posedness and ill-posedness for the Boltzmann equation. This research direction draws on a connection with dispersive equations. Under the inverse Fourier transform in the velocity variable, the free transport operator becomes a hyperbolic Schr\"{o}dinger operator.
By first introducing dispersive techniques developed for the study of quantum many-body hierarchy dynamics, notably space-time collapse estimates (see, e.g., \cite{CHPS15,CP10,Che12,KM08,CH16,CH19,HS19,KSS11}),
T. Chen-Denlinger-Pavlovi$\acute{\text{c}}$ \cite{CDP19} managed to push the
regularity for local well-posedness down to $s>1$ in the
Maxwell molecule and hard-sphere cases.
  For further developments employing dispersive and harmonic analysis techniques, see, for instance, \cite{BSTW26,CDP21,CDP24,CSZ23sharp,CSZ26}.
On the other hand, ill-posedness phenomena for dispersive
equations have by now been widely explored. See for example \cite{BT06,CCT03,KPV01,MST01,MST02,Nak99}. However, for the Boltzmann equation, despite certain similarities with dispersive equations, the potential ill-posed mechanisms remain nontrivial because of the complexity of the nonlinear collision kernel. In this direction,
Chen-Holmer \cite{CH24well} identified the well/ill-posedness threshold for the Boltzmann equation with constant kernel, while Chen-Shen-Zhang \cite{CSZ24well} carried this investigation further to cutoff
soft potentials. The
ill-posedness mechanism for the physically important hard-sphere model was
subsequently resolved by Chen-Guo-Shen-Zhang \cite{CGSZ26}, in which a strong-to-weak type ill-posedness result in the low-regularity regime $s<1$ was established.
This work shows that the ill-posedness of the hard-sphere model does not originate from the rapid growth of the collision kernel at large velocities, but is instead jointly driven by the loss term and dispersive effects, thereby revealing a dispersion-driven, loss-term-dominated nonlinear instability mechanism for the hard-sphere Boltzmann equation.
These developments suggest that transport
geometry and the asymmetry between the gain and loss terms may also be central to understanding the stability and instability phenomenon near equilibrium.

In the present work, we investigate the
hard-sphere Boltzmann equation near a global Maxwellian, with the aim of identifying
the threshold between stability and instability in a spatial Sobolev framework
and of understanding the mechanism that determines it.

To proceed, we denote a global Maxwellian by
\begin{align*}
\mu(v)=e^{-2|v|^{2}},
\end{align*}
and introduce the standard perturbation
\begin{align*}
F=\mu+f=\mu+\mu^{\frac{1}{2}}g.
\end{align*}
The perturbation equations for $f$ and $g$ then read
\begin{equation}
\left\{
\begin{aligned}\label{equ:perturbation,Boltzmann,f}
\pa_{t}f+v\cdot\nabla_{x}f+\nu f-K_{\mu}f=&Q(f,f),\\
f(0,x,v)=&f^{\mathrm{in}}(x,v),
\end{aligned}
\right.
\end{equation}
and
\begin{equation}
\left\{
\begin{aligned}\label{equ:perturbation,Boltzmann,g}
\pa_{t}g+v\cdot\nabla_{x}g+Lg=&\Gamma(g,g),\\
g(0,x,v)=&g^{\mathrm{in}}(x,v).
\end{aligned}
\right.
\end{equation}
In the above, the linearized operators $L$, $K$, $K_{\mu}$, and bilinear operator $\Gamma(g_{1},g_{2})$ are given by
\begin{align}
Lg=&\nu g-Kg,\quad \nu(v)=2\pi \int_{\R^3}|v-u|e^{-|u|^2}du,\label{equ:notation,L,K}\\
Kg=&\mu^{-\frac{1}{2}}\lrc{Q(\mu^{\frac{1}{2}}g,\mu)+Q^{+}(\mu,\mu^{\frac{1}{2}}g)},\label{equ:notation,K}\\
K_{\mu}f=&\mu^{\frac{1}{2}}K(\mu^{-\frac{1}{2}}f)=Q(f,\mu)+Q^{+}(\mu,f),\label{equ:notation,Kmu}\\
\Gamma(g_{1}, g_{2})=&\mu^{-\frac{1}{2}}Q(\mu^{\frac{1}{2}}g_{1},\mu^{\frac{1}{2}}g_{2}).\label{equ:notation,Gamma}
\end{align}

We first state an instability theorem for the perturbation equation
\eqref{equ:perturbation,Boltzmann,f}.
\begin{theorem}[Instability]\label{thm:main theorem,instability}
The perturbation equation \eqref{equ:perturbation,Boltzmann,f} is ill-posed in $e^{\lra{v}^{2}}L_{v}^{2}H_{x}^{s_{0}}$ for $s_{0}\in [0,1)$,
 in the sense that the data-to-solution map is not uniformly continuous.
More specifically, for each $M\gg1 $, there exist a time sequence
$\left\{T_{*}^{M}\right\}_{M}$ satisfying
$$T_{*}^{M}>0, \quad \lim_{M\to \infty}T_{*}^{M}=0,$$
and two solutions $f_{1}^{M}(t)$,  $f_{2}^{M}(t)$ in $[0,T_{*}^{M}]$ with
\begin{align}\label{equ:initial,data,f}
\n{e^{\lra{v}^{2}}f_{1}^{M}(0)}_{L_{v}^{2}H_{x}^{s_{0}}}\sim \n{e^{\lra{v}^{2}}f_{2}^{M}(0)}_{L_{v}^{2}H_{x}^{s_{0}}} \sim 1.
\end{align}
These solutions are initially close at $t=0$ in the Gaussian-weighted (strong) norm
\begin{equation*}
\n{e^{\lra{v}^{2}}\lrs{f_{1}^{M}(0)-f_{2}^{M}(0)}}_{L_{v}^{2}H_{x}^{s_{0}}}\lesssim \frac{1}{\ln M},
\end{equation*}
	but become fully separated at $t=T_{*}^{M}$ in the unweighted (weak) norm
\begin{equation*}
\n{f_{1}^{M}(T_{*}^{M})-f_{2}^{M}(T_{*}^{M})}_{L_{v}^{2}H_{x}^{s_{0}}}\gtrsim 1.
\end{equation*}
In other words, we have
\begin{align}\label{equ:instability,f1,f2}
\n{f_{1}^{M}(T_{*}^{M})-f_{2}^{M}(T_{*}^{M})}_{L_{v}^{2}H_{x}^{s_{0}}}\gtrsim \ln M \n{e^{\lra{v}^{2}}\lrs{f_{1}^{M}(0)-f_{2}^{M}(0)}}_{L_{v}^{2}H_{x}^{s_{0}}}.
\end{align}
\end{theorem}
We remark that,
under small perturbations of the initial data, the instability mechanism in Theorem \ref{thm:main theorem,instability} remains valid.
The initial data are taken to be of size $O(1)$, mainly in order to state the instability
estimate \eqref{equ:instability,f1,f2} in a clean form that emphasizes the
separation between $f_{1}^{M}$ and $f_{2}^{M}$.
In fact, we may instead take these initial data with sufficiently small norm
$\varepsilon_{0}\ll1$,  in which case the constant in the instability estimate
\eqref{equ:instability,f1,f2} then depends on $\varepsilon_{0}$.

Next, we turn to the global stability problem. As usual,
we work with the perturbation equation \eqref{equ:perturbation,Boltzmann,g} for $g$ and introduce the macroscopic and microscopic parts, which are defined by
\begin{align*}
\mathbf{P}_{0} g=:&\left\{a+\mathbf{b} \cdot v+c\frac{|v|^2-3}{\sqrt{6}}\right\} \mu^{\frac{1}{2}},\quad  \mathbf{P}_{1} g=:g-\mathbf{P}_{0} g.
\end{align*}
We are now ready to state the global stability result.

\begin{theorem}[Global Stability]\label{theorem:global stability}
Let $s>1$.
There are constants $C_{0}$ and $\ve_{0}>0$, such that if
\begin{align*}
\n{g^{\mathrm{in}}}_{L_{v}^{2}H_{x}^{s}}\leq \ve_{0},
\end{align*}
there exists a unique global solution $g(t,x,v)$ to \eqref{equ:perturbation,Boltzmann,g} such that
\begin{align*}
\mathcal{E}(g)+\mathcal{N}(g)\leq C_{0}\n{g^{\mathrm{in}}}_{L_{v}^{2}H_{x}^{s}},
\end{align*}
where
\begin{align}
\mathcal{E}(g):=&\n{g}_{L_{t}^{\infty}L_{v}^{2}H_{x}^{s}}+\bn{\lra{v}^{\frac{1}{2}}\mathbf{P}_{1}g}_{L_{t}^{2}L_{v}^{2}H_{x}^{s}}+
\bbn{\frac{|\nabla_{x}|}{\lra{\nabla_{x}}}\mathbf{P}_{0}g}_{L_{t}^{2}L_{v}^{2}H_{x}^{s}},\label{equ:E,g,energy,theorem}\\
\mathcal{N}(g):=& \n{\nu^{-\frac{1}{2}}\Gamma(g,g)}_{L_{t}^{2}L_{v}^{2}H_{x}^{s}}.
\label{equ:N,g,energy,theorem}
\end{align}

In particular, there exists a unique global solution $f(t,x,v)$ to \eqref{equ:perturbation,Boltzmann,f}
\begin{align}\label{equ:global,f,f0}
\n{e^{\lra{v}^{2}}f(t)}_{L_{t}^{\infty}L_{v}^{2}H_{x}^{s}}\lesssim
\n{e^{\lra{v}^{2}}f^{\mathrm{in}}}_{L_{v}^{2}H_{x}^{s}}.
\end{align}
Furthermore, the solution map is global-in-time Lipschitz continuous from $e^{\lra{v}^{2}}L_{v}^{2}H_{x}^{s}$ to
$e^{\lra{v}^{2}}L_{v}^{2}H_{x}^{s}$.
\end{theorem}

Theorems \ref{thm:main theorem,instability} and
\ref{theorem:global stability} show that the regularity threshold
between stability and instability lies precisely at $s=1$, thus giving
a complete characterization of stability near equilibrium.
The perturbation equation of the hard-sphere
Boltzmann equation around the Maxwellian exhibits rich and intricate structures, including the dispersive effect of the linear transport operator, the dissipation effect arising from the linearized operator, the large-velocity growth induced by the collision kernel, and the nonlocal and asymmetric nature of the gain
and loss terms. Within this intricate interplay of competing
mechanisms, the instability established in Theorem
\ref{thm:main theorem,instability} is in fact jointly driven by the
dispersive effect of the transport operator and the structure of the
loss term. This also provides an insight into the instability
mechanism, revealing how
the stabilizing dissipation effect can be overwhelmed by exploiting the
asymmetry between the gain and loss terms of the collision kernel.

In contrast, the sharp regularity in Theorem \ref{theorem:global stability} is achieved mainly through the interplay between the dissipative effect
of the linearized operator and the dispersive effect.
Although the dissipation effect around equilibrium has been extensively
studied, the novelty of the present result lies in incorporating the dispersive
effect into the analysis around equilibrium for the first time, leading to a sharp global stability result.

\subsection{Outline of the Proof}\label{sec:Outline of the Proof}
The proof of Theorem \ref{thm:main theorem,instability} relies on constructing two sequences of exact solutions to the perturbation Boltzmann equation \eqref{equ:perturbation,Boltzmann,f}, whose initial data are arbitrarily close in a weighted norm but become significantly separated within an arbitrarily short time.

A natural strategy is to pursue the recent line of investigation on ill-posedness mechanisms caused by the structure of the Boltzmann equation itself, as developed in \cite{CH24well,CSZ24well,CGSZ26}.
The near-equilibrium regime, however, is essentially different. The prominent feature is
the dissipation effect, which originates from the linearized collision
operator, has a stabilizing effect on the perturbations. We find that, the new ideas introduced in \cite{CGSZ26}, including the uniformly localized velocity and the forward-solving strategy introduced therein, can actually suppress this dissipation mechanism effectively. In other words, the ill-posedness mechanism remains dominant in the presence of dissipation, which makes the instability near the equilibrium possible.

To realize this instability, we advance the effort of \cite{CGSZ26} and pioneer the investigation of the near-equilibrium regime.
The key goal is to construct an exact solution $f^{M}(t)$ exhibiting pathological behavior, in the sense that its weighted norm decays rapidly over an arbitrarily short time interval. More precisely, we start from an exact solution $f^{M}(t)$ and split it
into an approximate part (ansatz) $f_{\mathrm{a}}^{M}(t)$ together with a correction term
\begin{align*}
f^{M}(t)=&f_{\mathrm{a}}^{M}(t)+f_{\mathrm{c}}^{M}(t).
\end{align*}

In Section \ref{sec:Approximate Solution and Norm Deflation}, we construct the approximate solution,
which admits the decomposition
\begin{align*}
f_{\mathrm{a}}^{M}(t)=f_{\mathrm{r}}^{M}(t)+f_{\mathrm{b}}^{M}(t),
\end{align*}
where $f_{\mathrm{r}}^{M}(t)$ serves as the principal term, while
$f_{\mathrm{b}}^{M}(t)$ a perturbation driving the ill-posed behavior. The principal term $f_{\mathrm{r}}^{M}(t)$ is required to satisfy
\begin{align}
&\partial_{t}f_{\mathrm{r}}^{M}=-Q^{-}(f_{\mathrm{r}}^{M},f_{\mathrm{b}}^{M})=-f_{\mathrm{r}}^{M}
A[f_{\mathrm{b}}^{M}],\label{equ:fr,equation,intro}
\end{align}
with
\begin{align*}
 A[f_{\mathrm{b}}^{M}]=\int_{\mathbb{R}^{3}}\int_{\mathbb{S}^{2}} f_{\mathrm{b}}^{M}(u) |(u-v)\cdot \omega| \,du d\omega.
\end{align*}
It follows from \eqref{equ:fr,equation,intro} that $f_{\mathrm{r}}^{M}$ indeed exhibits a rapid decay behaviour,
provided that the initial data $f_{\mathrm{r}}^{M}(0)$ is properly constructed
and $A[f_{\mathrm{b}}^{M}]$ is suitably controlled.

The key point here also lies in the construction of $f_{\mathrm{b}}^{M}$. Based on the features of the perturbation equation  \eqref{equ:perturbation,Boltzmann,f} near equilibrium, we require $f_{\mathrm{b}}^{M}$ to satisfy the damped transport equation
\begin{align}
\partial _{t}f_{\mathrm{b}}^{M}+v\cdot \nabla_{x}f_{\mathrm{b}}^{M}+\nu f_{\mathrm{b}}^{M}=0,\label{equ:fb,equation,intro}
\end{align}
The damping term $\nu f_{\mathrm{b}}^{M}$ is a structure specific to the near-equilibrium case,
and cannot be treated simply as an error term. In fact, in other norm scales, $\nu f_{\mathrm{b}}^{M}$ could behave as a large growth term with respect to the parameter $M$.

Such a damping, however, brings about a stabilization effect, forcing
$f_{\mathrm{b}}^{M}(t)$ to decay in time and thus disabling the mechanism that
drives the pathological behaviour.
This difficulty is overcome in the construction of the term
$f_{\mathrm{b}}^{M}(t)$ by incorporating a uniform localization in velocity.
Owing to this feature, the damping effect is kept under effective control on the
$O(1)$ time scale, namely $e^{-\nu t}\sim 1$. Consequently, the ill-posedness mechanism induced by the loss term continues to occupy a dominant position throughout the construction of the approximate solution.

With the approximate solution at hand, a further difficulty lies in the
treatment of the correction term $f_{\mathrm{c}}^{M}(t)$. In view of the
construction of the approximate solution, it is straightforward to verify that $f_{\mathrm{c}}^{M}(t)$
satisfies the nonlinear equation
\begin{align}\label{equ:correction term,fc,intro}
&\pa_{t}f_{\mathrm{c}}^{M} + v\cdot \nabla_{x} f_{\mathrm{c}}^{M}+\nu f_{\mathrm{c}}^{M}\\
=&\pm Q^\pm(f_{\mathrm{c}}^{M},f_{\mathrm{a}}^{M}) \pm
Q^\pm(f_{\mathrm{a}}^{M},f_{\mathrm{c}}^{M}) \pm Q^\pm(f_{\mathrm{c}}^{M},f_{\mathrm{c}
}^{M})+ K_{\mu}f_{\mathrm{c}}^{M}-F_{\text{err}}^{(1)}-F_{\text{err}}^{(2)},\notag
\end{align}
where the error terms are given by
\begin{align}\label{equ:Ferr,error term estimate,intro}
F_{\mathrm{err}}^{(1)}=&\pa_{t}f_{\mathrm{a}}^{M}+v\cdot \nabla_{x}f_{\mathrm{a}}^{M}+Q^{-}(f_{\mathrm{a}}^{M},f_{\mathrm{a}}^{M})-Q^{+}(f_{\mathrm{a}}^{M},f_{\mathrm{a}}^{M})+\nu f_{\mathrm{b}}^{M}\\
=&v\cdot \nabla_{x}\,f_{\mathrm{r}}^{M}-Q^{+}(f_{\mathrm{r}}^{M},f_{\mathrm{b}}^{M})\mp Q^{\pm}(f_{\mathrm{b}}^{M},f_{\mathrm{r}}^{M})\mp Q^{\pm}(f_{\mathrm{r}}^{M},f_{
\mathrm{r}}^{M})\mp Q^{\pm}(f_{\mathrm{b}}^{M},f_{\mathrm{b}}^{M}),\notag\\
F_{\mathrm{err}}^{(2)}=&\nu f_{\mathrm{r}}^{M}-K_{\mu}f_{\mathrm{r}}^{M}-K_{\mu}f_{\mathrm{b}}^{M}\\
=&\mp Q^{\pm}(f_{\mathrm{r}}^{M},\mu)\pm Q^{\pm}(\mu,f_{\mathrm{r}}^{M})-Q^{+}(f_{\mathrm{b}}^{M},\mu)-Q^{+}(\mu,f_{\mathrm{b}}^{M})
+Q^{-}(\mu,f_{\mathrm{b}}^{M})\notag.
\end{align}

In order to construct the exact solution, we need to establish the existence of the correction term $f_{\mathrm{c}}^{M}(t)$ and prove that it remains negligible as $M\to\infty$. Accordingly, we introduce the $Z$-norm
\begin{align}\label{equ:Z-norm,error term estimate,intro}
\n{f}_{Z}=M^{s_{0}}\n{f}_{L_{v}^{2,1}L_{x}^{2}}+M^{s_{0}-1}\n{\nabla_{x}f}_{L_{v}^{2,1}L_{x}^{2}}
+\n{f}_{L_{v}^{1,1}L_{x}^{\infty}}+
M^{-1}\n{\nabla_{x}f}_{L_{v}^{1,1}L_{x}^{\infty}}.
\end{align}
with $\n{f}_{L_{v}^{p,1}}=\n{\lra{v}f}_{L_{v}^{p}}$, which provides the working space for the perturbation argument.

The next key step is to prove that these error terms, formed jointly by the approximate solution and the equilibrium through the collision operator, are indeed negligible with respect to the parameter $M$ in the tailored $Z$-norm.
However, this is far from trivial. Let us illustrate with the equilibrium.
The equilibrium itself yields no decaying smallness, so the smallness must come from the approximate solution. Furthermore, the collision structure is important, since a direct example is that the damping term $Q^{-}(f_{\mathrm{b}}^{M},\mu)(t)=\nu f_{\mathrm{b}}^{M}(t)$ can indeed display large growth in $M$. Consequently, the uniform-in-time $Z$-norm estimates for the error terms rely on the analysis of the collision structure and on the interaction between the equilibrium and the approximate solution.

Specifically, in Section \ref{section:Estimates for Error Terms}, we focus on establishing uniform-in-time $Z$-norm
estimates for the error terms
\begin{equation}\label{equ:Ferr_bound,1,intro}
\bbn{e^{\lra{v}^{2}}\int_{t_{0}}^{t}S_{\nu}(t-\tau)F_{\text{err}}^{(1)}(\tau)\,d\tau}_{Z}\lesssim M^{-\delta},
\end{equation}
\begin{equation}\label{equ:Ferr_bound,2,intro}
\bbn{e^{\lra{v}^{2}}\int_{t_{0}}^{t}S_{\nu}(t-\tau)F_{\text{err}}^{(2)}(\tau)\,d\tau}_{Z}\lesssim M^{-\delta},
\end{equation}
where $S_{\nu}(t):=e^{-t\nu}e^{-tv\cdot \nabla_{x}}$ denotes the damped
transport operator. Building upon \eqref{equ:Ferr_bound,1,intro}--\eqref{equ:Ferr_bound,2,intro}, in Section \ref{section:Z-norm Estimates for Correction Terms} we further
derive the following uniform-in-time $Z$-norm estimates for the correction term
\begin{equation}\label{E:fc_bound2,intro}
\sup_{t\in[0,T_{*}]}\n{e^{(1-\kappa t)\lra{v}^{2}}f_{\mathrm{c}}^{M}(t)}_{Z}
\lesssim M^{-\frac{\delta}{2}}.
\end{equation}
Finally, in Section \ref{section:Proof of Instability}, we complete the
proof of instability Theorem
\ref{thm:main theorem,instability}.

In Section
\ref{sec:Low-regularity Global Stability},
we get into the analysis of the low-regularity global stability problem. Substantial progress on the global stability problem has been made under regularity assumptions that are as low as possible, for instance in \cite{Duan08,DLX16,Guo04,LYY04,NI77,Uka74}. The results available so far, still require the regularity $s\geq \frac{3}{2}$.
To reach the sharp regularity
threshold, our main idea is to integrate the dispersive effects into the
dissipative framework, notably the energy method developed by Guo
\cite{Guo04}.

 In Section \ref{section:Dissipation Energy Estimates}, we review the standard dissipative energy estimates
\begin{align}\label{equ:dissipation energy estimate,g,intro}
\mathcal{E}(g)
\lesssim& \n{g(0)}_{L_{v}^{2}H_{x}^{s}}+\mathcal{N}(g).
\end{align}
where the energy $\mathcal{E}(g)$ and $\mathcal{N}(g)$ is given by
\begin{align}
\mathcal{E}(g):=&\n{g}_{L_{t}^{\infty}L_{v}^{2}H_{x}^{s}}+
\bn{\lra{v}^{\frac{1}{2}}\mathbf{P}_{1}g}_{L_{t}^{2}L_{v}^{2}H_{x}^{s}}+
\bbn{\frac{|\nabla_{x}|}{\lra{\nabla_{x}}}\mathbf{P}_{0}g}_{L_{t}^{2}L_{v}^{2}H_{x}^{s}},\label{equ:E,g,energy,intro}\\
\mathcal{N}(g):=&\n{\lra{v}^{-\frac{1}{2}}(\pa_{t}+v\cdot \nabla_{x}+L)g}_{L_{t}^{2}L_{v}^{2}H_{x}^{s}}=\n{\lra{v}^{-\frac{1}{2}}\Gamma(g,g)}_{L_{t}^{2}L_{v}^{2}H_{x}^{s}}.
\label{equ:N,g,energy,intro}
\end{align}
The dissipation energy estimates \eqref{equ:dissipation energy estimate,g,intro} have been investigated by many authors under various settings, for instance in \cite{Duan08,Duan11,DLS25,Guo03,Guo03vlasov,Guo04,Guo10,LYY04}.

Due to the dissipative energy estimates \eqref{equ:dissipation energy estimate,g,intro}, it remains to establish the $L_{t}^{2}$ estimate for the nonlinear term.
This is different from the near-vacuum setting. In our previous work \cite{CSZ23sharp}, we established the
$L_{t}^{1}$ estimate through the Duhamel expansion together with the dispersive effect of the transport operator. Therefore, we first express
$g$ in the Duhamel form associated with the damped transport operator
\begin{align}
g(t)=S_{\nu}(t)g(0)+\int_{0}^{t}S_{\nu}(t-s)Kg(s)ds+
\int_{0}^{t}S_{\nu}(t-s)\Gamma(g,g)(s)ds.
\end{align}
Within this decomposition, the initial part $S_{\nu}(t)g(0)$ and the nonlinear term $\Gamma(g,g)$ can be treated by
exploiting the dispersive effects of the damped transport operator. However, the main point is to absorb the linear part $Kg$. The key idea is to prove dispersive-type bilinear estimates for the Duhamel integrals and combine them with the dissipative estimates so that the linear part $Kg$ is absorbed into the energy $\mathcal{E}(g)$.

Specifically, in Section \ref{section:Dispersive-type
Bilinear Estimates}, we establish dispersive-type bilinear estimates for the nonlinear term as follows
\begin{align*}
\n{\lra{v}^{-\frac{1}{2}}\Gamma^{\pm}(S_{\nu}(t)g_{1}(0),S_{\nu}(t)g_{2}(0))}_{L_{t}^{2}L_{v}^{2}H_{x}^{s}}\lesssim & \n{g_{1}(0)}_{L_{v}^{2}H_{x}^{s}}\n{g_{2}(0)}_{L_{v}^{2}H_{x}^{s}},\\
\bbn{\lra{v}^{-\frac{1}{2}}\Gamma^{\pm}\lrs{\int_{0}^{t}S_{\nu}(t-t_{1})g_{1}(t_{1})
dt_{1},S_{\nu}(t)g_{2}(0)}}_{L_{t}^{2}L_{v}^{2}H_{x}^{s}}
\lesssim &
\bn{\lra{v}^{-\frac{1}{2}}
g_{1}(t_{1})}_{L_{t_{1}}^{2}L_{v}^{2}H_{x}^{s}}\n{g_{2}(0)}_{L_{v}^{2}H_{x}^{s}},\\
\bbn{\lra{v}^{-\frac{1}{2}}\Gamma^{\pm}\lrs{S_{\nu}(t)g_{1}(0),\int_{0}^{t}S_{\nu}(t-t_{2})
g_{2}(t_{2})dt_{2}}}_{L_{t}^{2}L_{v}^{2}H_{x}^{s}}
\lesssim &
\n{g_{1}(0)}_{L_{v}^{2}H_{x}^{s}}
\bn{\lra{v}^{-\frac{1}{2}}
g_{2}(t_{2})}_{L_{t_{2}}^{2}L_{v}^{2}H_{x}^{s}},
\end{align*}
\begin{align*}
&\bbn{\lra{v}^{-\frac{1}{2}}\Gamma^{\pm}\lrs{\int_{0}^{t}S_{\nu}(t-t_{1})g_{1}(t_{1})
dt_{1},\int_{0}^{t}S_{\nu}(t-t_{2})g_{2}(t_{2})dt_{2}}}_{L_{t}^{2}L_{v}^{2}H_{x}^{s}}\\
\lesssim &
\bn{\lra{v}^{-\frac{1}{2}}
g_{1}(t_{1})}_{L_{t_{1}}^{2}L_{v}^{2}H_{x}^{s}}\bn{\lra{v}^{-\frac{1}{2}}
g_{2}(t_{2})}_{L_{t_{2}}^{2}L_{v}^{2}H_{x}^{s}}.\notag
\end{align*}
These estimates
serve as the key tool for pushing the regularity down to the sharp
level $s>1$.

With these dissipation energy estimates and dispersive-type bilinear estimates in place, the proof of global stability Theorem \ref{theorem:global stability} is carried out in Section \ref{section:Proof of Global Stability}.

\section{Instability}
In this section, we continue the work \cite{CGSZ26} for the ill-posedness mechanism of
the hard-sphere Boltzmann equation to investigate the instability mechanism near equilibrium.
We construct two sequences of exact solutions to the perturbed
Boltzmann equation \eqref{equ:perturbation,Boltzmann,f}. One sequence exhibits
the pathological behavior, while the other remains stable near its initial
datum.

In Section \ref{sec:Approximate Solution and Norm Deflation}, we construct the
approximate solution satisfying the norm deflation property and provide Sobolev norm estimates.
In Section \ref{section:Estimates for Error Terms}, we introduce a tailored
$Z$-norm, which provides the working space for the perturbation argument that
turns the approximate solution into an exact one, and establish uniform-in-time
$Z$-norm estimates for the error terms.
In Section \ref{section:Z-norm Estimates for Correction Terms}, we derive
uniform-in-time $Z$-norm estimates for the correction term. In addition, we
construct the second solution $g(t)$, which is stable near its initial
datum.
In Section \ref{section:Proof of Instability}, we complete the proof of
Theorem \ref{thm:main theorem,instability}.

\subsection{Approximate Solution and Norm Deflation} \label{sec:Approximate Solution and Norm Deflation}
We decompose the approximate solution $f_{\mathrm{a}}$ into two parts
\begin{align*}
f_{\mathrm{a}}(t)=f_{\mathrm{r}}(t)+f_{\mathrm{b}}(t).
\end{align*}
Here, for notational simplicity, we omit the superscript parameter $M$ throughout.

Based on the specific structure of the perturbation equation \eqref{equ:perturbation,Boltzmann,f},
we construct $f_{\mathrm{b}}$
to satisfy the damped transport equation
\begin{align}\label{equ:fb equation}
\pa_{t}f_{\mathrm{b}}+v\cdot \nabla_{x}f_{\mathrm{b}}+\nu f_{\mathrm{b}}=0.
\end{align}
Consequently, we obtain the explicit formula
\begin{align*}
f_{\mathrm{b}}(t)=e^{-\nu t}f_{\mathrm{b}}(0,x-vt,v).
\end{align*}

For the construction $f_{\mathrm{b}}$, we employ a lattice system that is roughly uniformly distributed on the unit hemisphere. More precisely, we take the points on the unit sphere defined by
 \begin{align}\label{equ:3d,eij}
 e_{i,j}=\lrs{\sin(\frac{\pi j}{2M}) \cos( \frac{2\pi i}{j}),\sin(\frac{\pi j}{2M})\sin(\frac{2\pi i}{j}),\cos (\frac{\pi j}{2M})},\quad 1\leq i\leq j\leq M.
 \end{align}
 These points are roughly uniformly spaced across the upper hemisphere, with the total number of points scaling as order $M^{2}$.
 The perturbation term $f_{\mathrm{b}}$ is then constructed as a superposition of tubes aligned along these lattice points in various directions, given by
 \begin{equation}\label{equ:fb,t}
f_{\mathrm{b}}(t,x,v)=M^{1-s}e^{-\nu t}\sum_{1\leq i\leq j\leq M}K_{i,j}(x-vt)I_{i,j}(v),
\end{equation}
where
\begin{align*}
K_{i,j}(x)&=\chi (MP_{e_{i,j}}^{\perp }x)\chi (P_{e_{i,j}}x), \\
I_{i,j}(v)&=\chi(MP_{e_{i,j}}^{\perp }v)\chi (10P_{e_{i,j}}(v-e_{i,j})).
\end{align*}
Here, $\chi(x)$ is a smooth cutoff function supported in the unit ball $B(0,1)$. We denote by $P_{e_{i,j}}$ the orthogonal projection onto the one-dimensional subspace spanned by $e_{i,j}$, and by $P_{e_{i,j}}^{\perp}$ the orthogonal projection onto the orthogonal complement $\lr{e_{i,j}}^{\perp}$.

On the other hand, the principle part $f_{\mathrm{r}}$ is designed to satisfy the drift-free linearized loss-term Boltzmann equation
\begin{align}\label{equ:fr,loss-term}
\partial_{t}f_{\mathrm{r}}=-Q^{-}(f_{\mathrm{r}},f_{\mathrm{b}}).
\end{align}
The initial datum of the principal term $f_{\mathrm{r}}$ is localized near the origin, that is,
$$f_{\mathrm{r}}(0,x,v)=M^{\frac{3}{2}-s_{0}}N^{\frac{3}{2}} \chi (100Mx)\chi(Nv).$$
Using \eqref{equ:fr,loss-term}, we may write the solution $f_{\mathrm{r}}(t)$ as
\begin{equation}\label{equ:fr}
f_{\mathrm{r}}(t,x,v)=M^{\frac{3}{2}-s_{0}}N^{\frac{3}{2}}\exp \left[
-\int_{0}^{t}A[f_{\mathrm{b}}](\tau,x,v)\,d\tau \right] \chi (100Mx)\chi (Nv),
\end{equation}
where the term $A[f_{\mathrm{b}}]$ is given, up to a constant, by
\begin{equation}\label{equ:Afb}
A[f_{\mathrm{b}}](t,x,v)\sim \int_{\R^{3}}|u-v|f_{\mathrm{b}}(t,x,u)du.
\end{equation}

\vspace{0.5em}
\paragraph{\textbf{Parameter selection}.}
The parameters are chosen as follows
\begin{align}
	&M\gg 1, \quad  N=M^{10},\label{equ:condition,parameters}\\
&\max\lr{s_{0},\frac{3}{4}}<s<1, \label{equ:condition,s,s0}\\
&T_{*}=M^{s-1}(\ln\ln M).\label{equ:T*}
\end{align}
We remark that the parameter $N$ is chosen to be $M^{10}$ merely for
convenience in the subsequent estimates. In fact, the more general condition
$N\gg M$ would suffice.

\vspace{0.5em}
\noindent\textbf{Sobolev Norm Estimates for Approximate Solutions.}
We review the geometry estimates of directional projections and the Sobolev norm estimates for
$f_{\mathrm{b}}$ and $f_{\mathrm{r}}$ from \cite{CGSZ26}. Compared with the construction in \cite{CGSZ26}, our $f_{\mathrm{b}}(t)$
contains an extra dissipative factor $e^{-\nu(v)t}$. Under the uniform compact support constraint for $f_{\mathrm{b}}(t)$, namely,
\begin{equation}\label{equ:uniform compact,fb}
\operatorname{Supp}f_{\mathrm{b}}(t)\subset B(0,10).
\end{equation}
this dissipative factor is negligible on the finite time scale $t\in[0,1]$, since $e^{-\nu(v)t}\sim 1$. Consequently, the Sobolev and pointwise upper and lower bounds for $f_{\mathrm{b}}(t)$, and hence
$f_{\mathrm{r}}(t)$, established in \cite{CGSZ26} also hold for our construction.

\begin{lemma}[Separation of directions and summation bounds, {\cite[Lemma 2.1]{CGSZ26}}]\label{lemma:eij,property}
 The following estimates hold
\begin{align}
|e_{i_{1},j_{1}}-e_{i_{2},j_{2}}|\geq& \frac{1}{M},\quad \text{for $(i_{1},j_{1})\neq (i_{2},j_{2})$},\label{equ:gap estimate,eij}\\
\sum_{i,j}|\nabla_{x}^{k}K_{i,j}(x)|\lesssim& M^{2+k}\lrs{\frac{1}{M|x|+1}}^{2}\chi(x),\quad k=0,1,\label{equ:Kij,sum}\\
\sum_{i,j}I_{i,j}(v)\leq& 1_{\lr{\frac{3}{4}\leq|v|\leq \frac{5}{4}}}(v).\label{equ:Iij,sum}
\end{align}
\end{lemma}

\begin{lemma}[Sobolev norm estimates for $f_{\mathrm{b}}$,  {\cite[Lemmas 2.2, 2.3]{CGSZ26}}]\label{lemma:Hs,bounds on fb}
For $t\in [0,\frac{1}{4}]$, we have
\begin{align}
\n{e^{\lra{v}^{2}}f_{\mathrm{b}}(t)}_{L_{v}^{2}H_{x}^{s_{0}}}
\lesssim & M^{s_{0}-s},\label{equ:critical norm estimate for fb}\\
\n{e^{\lra{v}^{2}}f_{\mathrm{b}}(t)}_{L_{v}^{p}L_{x}^{\infty}}
+M^{-1}\n{e^{\lra{v}^{2}}\nabla_{x}f_{\mathrm{b}}(t)}_{L_{v}^{p}L_{x}^{\infty}}\lesssim& M^{1-s},\quad p\in[1,\infty],
\label{equ:sobolev norm,fb,LvpLxinf}\\
\n{e^{\lra{v}^{2}}f_{\mathrm{b}}(t)}_{L_{x}^{2}L_{v}^{1}}+
M^{-1}\n{e^{\lra{v}^{2}}\nabla_{x}f_{\mathrm{b}}(t)}_{L_{x}^{2}L_{v}^{1}}\lesssim& M^{-\frac{1}{2}-s}.\label{equ:sobolev norm,fb,L2L1}
\end{align}

For $k=0,1,2$, we have the pointwise upper bound
\begin{align}\label{equ:upper bound,fr,k}
\big|\chi(Nv) \nabla_{x}^{k}A[f_{\mathrm{b}}](t,x,v)\big |\lesssim& M^{k+1-s},\\
A[f_{\mathrm{b}}]\lesssim& M^{1-s}\lrs{\frac{1}{M|x|+1}}^{2}.\label{equ:upper bound,Afb}
\end{align}
Moreover, we have the corresponding lower bound
\begin{equation}\label{equ:upper lower bound,fr}
\chi(Nv)A[f_{\mathrm{b}}] \gtrsim  M^{1-s} \chi(Nv),\quad \text{for $|x|\leq \frac{1}{100M}$}.
\end{equation}
\end{lemma}

\begin{lemma}[Sobolev norm estimates for $f_{\mathrm{r}}$,  {\cite[Lemmas 2.4, 4.1]{CGSZ26}}]\label{lemma:bounds on fr}
For $t\in[0,\frac{1}{4}]$, we have the upper bound
\begin{align}
\n{e^{\lra{v}^{2}}f_{\mathrm{r}}(t)}_{L_{v}^{2}H_{x}^{s_{0}}}\lesssim&  \exp[-t M^{1-s}],\label{equ:critical norm estimate for fr,upper bound}\\
\n{f_{\mathrm{r}}(0)}_{L_{v}^{2}H_{x}^{s_{0}}}\gtrsim&  1.\label{equ:critical norm estimate for fr,lower bound}
\end{align}
In particular, the following estimates hold
\begin{align}
\n{e^{\lra{v}^{2}}f_{\mathrm{r}}(0)}_{L_{v}^{2}H_{x}^{s_{0}}}\sim & \n{f_{\mathrm{r}}(0)}_{L_{v}^{2}H_{x}^{s_{0}}}\sim 1,\label{equ:upper bound,fr,t=0}\\
\n{e^{\lra{v}^{2}}f_{\mathrm{r}}(T_{*})}_{L_{v}^{2}H_{x}^{s_{0}}}\lesssim & \frac{1}{ \ln M},\label{equ:lower bound,fr,T}
\end{align}
where $T_{*}=M^{s-1}(\ln\ln M)$. Moreover, for $p\in [1,\infty]$, we have
\begin{align}
\n{\lra{v}^{2}e^{\lra{v}^{2}}f_{\mathrm{r}}}_{L_{t}^{\infty}(0,T_{*};L_{v}^{1}L_{x}^{p})}\lesssim & N^{-1},\label{equ:z-norm,lv1,used}\\
M^{-1}\n{\lra{v}^{2}e^{\lra{v}^{2}}\nabla_{x}f_{\mathrm{r}}}_{L_{t}^{\infty}(0,T_{*};L_{v}^{1}L_{x}^{p})}
\lesssim& N^{-1}.\label{equ:z-norm,lv1,H1,used}
\end{align}
\end{lemma}
\begin{proof}
The estimates \eqref{equ:critical norm estimate for fr,upper
bound}--\eqref{equ:lower bound,fr,T} are direct consequences of \cite[Lemma
2.4]{CGSZ26}, while \eqref{equ:z-norm,lv1,used} and
\eqref{equ:z-norm,lv1,H1,used} are obtained by adapting the proof of
\cite[Lemma 4.1]{CGSZ26}. In the latter, the vanishing smallness factor
$N^{-1}$ stems from the $L^{1}_{v}$ integration in the velocity variable,
combined with the parameter choice $N=M^{10}$. These two estimates
will be used in the subsequent error analysis.
\end{proof}

Based on the above Sobolev norm estimates for $f_{\mathrm{b}}(t)$ and $f_{\mathrm{r}}(t)$, we are ready to conclude the norm deflation property for the approximate solution $f_{\mathrm{a}}(t)$.
\begin{proposition}[Norm deflation property for $f_{\mathrm{a}}$]\label{lemma:norm deflation of fa}
Let $T_{*}= M^{s-1}(\ln\ln M)$. Then, we have
\begin{align}
	\n{e^{\lra{v}^{2}}f_{\mathrm{a}}(0)}_{L_v^{2} H_x^{s_{0}}}\sim&  1,\label{equ:critical norm estimate for fa,t=0}\\
	\n{e^{\lra{v}^{2}}f_{\mathrm{a}}(T_{*})}_{L_v^{2} H_x^{s_{0}}}\lesssim& \frac{1}{\ln M}.\label{equ:critical norm estimate for fa,t=T}
\end{align}
\end{proposition}
\begin{proof}
Since $f_{\mathrm{b}}$ and $f_{\mathrm{r}}$ have disjoint velocity supports, the norm of their sum satisfies
\begin{align*}
\n{e^{\lra{v}^{2}}f_{\mathrm{a}}(t)}_{L_v^{2} H_x^{s_{0}}}^{2}\sim \n{e^{\lra{v}^{2}}f_{\mathrm{r}}(t)}_{L_v^{2} H_x^{s_{0}}}^{2}+\n{e^{\lra{v}^{2}}f_{\mathrm{b}}(t)}_{L_v^{2} H_x^{s_{0}}}^{2}.
\end{align*}
The estimate for $f_{\mathrm{b}}$ is provided by Lemma \ref{lemma:Hs,bounds on fb}, specifically
\begin{align*}
\sup_{t\in[0,T_{*}]}\n{e^{\lra{v}^{2}}f_{\mathrm{b}}(t,x,v)}_{L_{v}^{2}H_{x}^{s_{0}}}\lesssim M^{s_{0}-s}\lesssim \frac{1}{\ln M}.
\end{align*}
For $f_{\mathrm{r}}$, from Lemma \ref{lemma:bounds on fr} we have
\begin{align*}
\n{e^{\lra{v}^{2}}f_{\mathrm{r}}(0)}_{L_{v}^{2}H_{x}^{s_{0}}}\sim & \n{f_{\mathrm{r}}(0)}_{L_{v}^{2}H_{x}^{s_{0}}}\sim 1,\\
\n{e^{\lra{v}^{2}}f_{\mathrm{r}}(T_{*})}_{L_{v}^{2}H_{x}^{s_{0}}}\lesssim & \frac{1}{\ln M}.
\end{align*}
Combining these above bounds, we directly obtain \eqref{equ:critical norm estimate for fa,t=0} and \eqref{equ:critical norm estimate for fa,t=T}.
\end{proof}

\subsection{$Z$-norm Estimates for Error Terms}\label{section:Estimates for Error Terms}
A key step toward perturbing the approximate solution $f_{\mathrm{a}}$ to an exact solution of the Boltzmann equation \eqref{equ:perturbation,Boltzmann,f} is  to develop a well-posedness theory for the correction equation in a suitable space, which takes the form
\begin{align}\label{equ:correction term,fc}
\pa_{t}f_{\mathrm{c}} + v\cdot \nabla_{x} f_{\mathrm{c}}+\nu f_{\mathrm{c}}
=&\pm Q^\pm(f_{\mathrm{c}},f_{ \mathrm{a}}) \pm
Q^\pm(f_{\mathrm{a}},f_{\mathrm{c}}) \pm Q^\pm(f_{\mathrm{c}},f_{\mathrm{c}
})+ K_{\mu}f_{\mathrm{c}}-F_{\text{err}}^{(1)}-F_{\text{err}}^{(2)},
\end{align}
where the error terms are given by
\begin{align}\label{equ:Ferr,error term estimate}
F_{\mathrm{err}}^{(1)}=&\pa_{t}f_{\mathrm{a}}+v\cdot \nabla_{x}f_{\mathrm{a}}+Q^{-}(f_{\mathrm{a}},f_{\mathrm{a}})-Q^{+}(f_{\mathrm{a}},f_{\mathrm{a}})+\nu f_{\mathrm{b}}\\
=&v\cdot \nabla
_{x}\,f_{\mathrm{r}}-Q^{+}(f_{\mathrm{r}},f_{\mathrm{b}})\mp Q^{\pm}(f_{\mathrm{b}},f_{\mathrm{r}})\mp Q^{\pm}(f_{\mathrm{r}},f_{
\mathrm{r}})\mp Q^{\pm}(f_{\mathrm{b}},f_{\mathrm{b}}),\notag\\
F_{\mathrm{err}}^{(2)}=&\nu f_{\mathrm{r}}-K_{\mu}f_{\mathrm{r}}-K_{\mu}f_{\mathrm{b}}\\
=&\mp Q^{\pm}(f_{\mathrm{r}},\mu)\pm Q^{\pm}(\mu,f_{\mathrm{r}})-Q^{+}(f_{\mathrm{b}},\mu)-Q^{+}(\mu,f_{\mathrm{b}})+Q^{-}(\mu,f_{\mathrm{b}})\notag.
\end{align}

Specifically, we introduce the $Z$-norm
\begin{align}\label{equ:Z-norm,error term estimate}
\n{f}_{Z}=M^{s_{0}}\n{f}_{L_{v}^{2,1}L_{x}^{2}}+M^{s_{0}-1}\n{\nabla_{x}f}_{L_{v}^{2,1}L_{x}^{2}}
+\n{f}_{L_{v}^{1,1}L_{x}^{\infty}}+
M^{-1}\n{\nabla_{x}f}_{L_{v}^{1,1}L_{x}^{\infty}}
\end{align}
with $\n{f}_{L_{v}^{p,1}}=\n{\lra{v}f}_{L_{v}^{p}}$, which provides the working space for the perturbation argument in Section \ref{section:Z-norm Estimates for Correction Terms}.

This section is devoted to deriving the $Z$-norm estimates for the error terms $F_{\mathrm{err}}^{(1)}$ and
$F_{\mathrm{err}}^{(2)}$.
The main result of this section is summarized in the following proposition.
\begin{proposition}[$Z$-norm estimates for $F_{\mathrm{err}}^{(1)}$ and
$F_{\mathrm{err}}^{(2)}$]\label{lemma:bounds on ferr}
 For $0\leq  t_{0}\leq t \leq T_{*}$, we have
\begin{equation}\label{equ:Ferr_bound,1}
\bbn{e^{\lra{v}^{2}}\int_{t_{0}}^{t}S_{\nu}(t-\tau)F_{\text{err}}^{(1)}(\tau)\,d\tau}_{Z}\lesssim M^{-\delta},
\end{equation}
\begin{equation}\label{equ:Ferr_bound,2}
\bbn{e^{\lra{v}^{2}}\int_{t_{0}}^{t}S_{\nu}(t-\tau)F_{\text{err}}^{(2)}(\tau)\,d\tau}_{Z}\lesssim M^{-\delta},
\end{equation}
for some $\delta>0$.
\end{proposition}
\begin{proof}
As for \eqref{equ:Ferr_bound,1}, the
$Z$-norm estimates in the case of the free transport operator $S(t)$ have already been derived
in \cite[Proposition 4.1]{CGSZ26}, and the corresponding arguments apply equally to the damped transport operator $S_{\nu}(t)$. Consequently, it remains to complete the proof of \eqref{equ:Ferr_bound,2}. This will be achieved in Lemmas \ref{lemma:Z-norm estimate,fr,mu}, \ref{lemma:Z-norm estimate,Q-,fb,mu}, and \ref{lemma:Z-norm estimate,Q+,fb,mu}, respectively.
\end{proof}

\subsubsection{Analysis of $Q^{\pm}(f_{\mathrm{r}},\mu)$, and $Q^{\pm}(\mu,f_{\mathrm{r}})$}\label{section:
Analysis,Q,fr,mu}
To begin with, we collect the collision estimates with respect to the $Z$-norm in the following Lemma \ref{lemma:weighted,Z-norm,estimates}
\begin{lemma}\label{lemma:weighted,Z-norm,estimates}
There holds that
\begin{align}
\n{e^{\lra{v}^{2}}Q^{+}(f,\mu)}_{L_{v}^{2,1}L_{x}^{2}}+
\n{e^{\lra{v}^{2}}Q^{+}(\mu,f)}_{L_{v}^{2,1}L_{x}^{2}}
\lesssim& \n{\lra{v}^{2}e^{\lra{v}^{2}}f}_{L_{x}^{2}L_{v}^{1}},\label{equ:Q+,sym,f,mu,Lv2Lx2}\\
\n{e^{\lra{v}^{2}}Q^{+}(f,\mu)}_{L_{v}^{1,1}L_{x}^{\infty}}+
\n{e^{\lra{v}^{2}}Q^{+}(\mu,f)}_{L_{v}^{1,1}L_{x}^{\infty}}\lesssim&
\n{\lra{v}^{2}e^{\lra{v}^{2}}f}_{L_{v}^{1}L_{x}^{\infty}},\label{equ:Q+,sym,f,mu,Lv1Lxinf}\\
\n{e^{\lra{v}^{2}}Q^{-}(\mu,f)}_{L_{v}^{2,1}L_{x}^{2}}
\lesssim& \n{\lra{v}^{2}e^{\lra{v}^{2}}f}_{L_{x}^{2}L_{v}^{1}},\label{equ:Q-,f,mu,Lv2Lx2}\\
\n{e^{\lra{v}^{2}}Q^{-}(\mu,f)}_{L_{v}^{1,1}L_{x}^{\infty}}
\lesssim& \n{\lra{v}^{2}e^{\lra{v}^{2}}f}_{L_{v}^{1}L_{x}^{\infty}}.\label{equ:Q-,f,mu,Lv1Lxinf}
\end{align}
In particular, for the $Z$-norm, we have
\begin{align}\label{equ:Z-norm,estimates,mu,f}
&\n{e^{\lra{v}^{2}}Q^{+}(\mu,f)}_{Z}+
\n{e^{\lra{v}^{2}}Q^{+}(f,\mu)}_{Z}+\n{e^{\lra{v}^{2}}Q^{-}(\mu,f)}_{Z}\\
\lesssim& M^{s_{0}}\n{\lra{v}^{2}e^{\lra{v}^{2}}f}_{L_{x}^{2}L_{v}^{1}}
+M^{s_{0}-1}\n{\lra{v}^{2}e^{\lra{v}^{2}}\nabla_{x}f}_{L_{x}^{2}L_{v}^{1}}\notag\\
&+\n{\lra{v}^{2}e^{\lra{v}^{2}}f}_{L_{v}^{1}L_{x}^{\infty}}+
M^{-1}\n{\lra{v}^{2}e^{\lra{v}^{2}}\nabla_{x}f}_{L_{v}^{1}L_{x}^{\infty}}.\notag
\end{align}
\end{lemma}
\begin{proof}
We first provide the following collision estimates for hard-sphere case
\begin{align}
\n{Q^{-}(f,g)}_{L_{v}^{r}}\lesssim& \n{f}_{L_{v}^{r,1}}\n{g}_{L_{v}^{1,1}},\quad \text{for $r\in[1,\infty]$},\label{equ:collision,estimate,Q-,Lr}\\
\n{Q^{+}(f,g)}_{L_{v}^{r}}\lesssim& \n{f}_{L_{v}^{r,1}}\n{g}_{L_{v}^{1,1}},\quad \text{for $r\in[1,3)$},\label{equ:collision,estimate,Q+,Lr}\\
\n{Q^{+}(f,g)}_{L_{v}^{r}}\lesssim& \n{f}_{L_{v}^{1,1}}\n{g}_{L_{v}^{r,1}},\quad \text{for $r\in[1,3)$}.\label{equ:collision,estimate,Q+,Lr,gr}
\end{align}
The estimate for the loss term \eqref{equ:collision,estimate,Q-,Lr} follows directly. Indeed, by bounding the relative velocity as $|u-v|\lesssim
\lra{u}\lra{v}$, we obtain
 \begin{align*}
\n{Q^{-}(f,g)}_{L_{v}^{r}}\lesssim& \bbn{f(v)\int_{\R^{3}}|u-v||g(u)|du}_{L_{v}^{r}}\\
\lesssim&\n{\lra{v}f(v)}_{L_{v}^{r}}\n{\lra{u}g(u)}_{L_{u}^{1}}.
 \end{align*}
 The gain term estimates \eqref{equ:collision,estimate,Q+,Lr}--\eqref{equ:collision,estimate,Q+,Lr,gr} follow from \cite[Theorem 1]{ACG10} or \cite[Theorem 2.1]{MV04}.

Then using the collision estimates \eqref{equ:collision,estimate,Q-,Lr} and \eqref{equ:collision,estimate,Q+,Lr,gr} with $r=2$, we derive
\begin{align}
&\n{e^{\lra{v}^{2}}Q^{+}(\mu,f)}_{L_{v}^{2,1}L_{x}^{2}}+
\n{e^{\lra{v}^{2}}Q^{+}(f,\mu)}_{L_{v}^{2,1}L_{x}^{2}}+
\n{e^{\lra{v}^{2}}Q^{-}(\mu,f)}_{L_{v}^{2,1}L_{x}^{2}}\\
\lesssim &
\bbn{Q^{\pm}\lrs{\lra{v}e^{\lra{v}^{2}}\mu,\lra{v}e^{\lra{v}^{2}}f}}_{L_{x}^{2}L_{v}^{2}}+
\bbn{Q^{+}\lrs{\lra{v}e^{\lra{v}^{2}}f,\lra{v}e^{\lra{v}^{2}}\mu}}_{L_{x}^{2}L_{v}^{2}}\notag\\
\lesssim& \n{\lra{v}^{2}e^{\lra{v}^{2}}\mu(v)}_{L_{v}^{2}}\n{\lra{v}^{2}e^{\lra{v}^{2}}f}_{L_{x}^{2}L_{v}^{1}}\lesssim
\n{\lra{v}^{2}e^{\lra{v}^{2}}f}_{L_{x}^{2}L_{v}^{1}}.\notag
\end{align}
which completes the proof of \eqref{equ:Q+,sym,f,mu,Lv2Lx2} and \eqref{equ:Q-,f,mu,Lv2Lx2}.

In a similar way, using again the collision estimates \eqref{equ:collision,estimate,Q-,Lr} and \eqref{equ:collision,estimate,Q+,Lr,gr} with $r=1$, we obtain
\begin{align}
&\n{e^{\lra{v}^{2}}Q^{+}(\mu,f)}_{L_{v}^{1,1}L_{x}^{\infty}}+
\n{e^{\lra{v}^{2}}Q^{+}(f,\mu)}_{L_{v}^{1,1}L_{x}^{\infty}}+
\n{e^{\lra{v}^{2}}Q^{-}(\mu,f)}_{L_{v}^{1,1}L_{x}^{\infty}}\\
\lesssim &
\bbn{Q^{\pm}\lrs{\lra{v}e^{\lra{v}^{2}}\mu,\lra{v}e^{\lra{v}^{2}}\n{f}_{L_{x}^{\infty}}}}_{L_{v}^{2}}+
\bbn{Q^{+}\lrs{\lra{v}e^{\lra{v}^{2}}\n{f}_{L_{x}^{\infty}},\lra{v}e^{\lra{v}^{2}}\mu}}_{L_{v}^{1}}\notag\\
\lesssim& \n{\lra{v}^{2}e^{\lra{v}^{2}}\mu(v)}_{L_{v}^{2}}\n{\lra{v}^{2}e^{\lra{v}^{2}}f}_{L_{v}^{1}L_{x}^{\infty}}\lesssim
\n{\lra{v}^{2}e^{\lra{v}^{2}}f}_{L_{v}^{1}L_{x}^{\infty}}.\notag
\end{align}
which completes the proof of \eqref{equ:Q+,sym,f,mu,Lv1Lxinf} and \eqref{equ:Q-,f,mu,Lv1Lxinf}.

The $Z$-norm estimate \eqref{equ:Z-norm,estimates,mu,f} then follows from \eqref{equ:Q+,sym,f,mu,Lv2Lx2}--\eqref{equ:Q-,f,mu,Lv1Lxinf}.
\end{proof}

With the collision estimates at hand, we are now in a position to analyze the error terms.
We recall the key bounds for $f_{\mathrm{r}}$ established in Lemma \ref{lemma:bounds on fr}, which read
\begin{align}
\n{\lra{v}^{2}e^{\lra{v}^{2}}f_{\mathrm{r}}}_{L_{t}^{\infty}(0,T_{*};Z)}\lesssim& 1,\label{equ:z-norm,fr,used}\\
\n{\lra{v}^{2}e^{\lra{v}^{2}}f_{\mathrm{r}}}_{L_{t}^{\infty}(0,T_{*};L_{v}^{1}L_{x}^{\infty})}\lesssim & N^{-1},\label{equ:z-norm,lv1,fr}\\
M^{-1}\n{\lra{v}^{2}e^{\lra{v}^{2}}\nabla_{x}f_{\mathrm{r}}}_{L_{t}^{\infty}(0,T_{*};L_{v}^{1}L_{x}^{\infty})}
\lesssim& N^{-1},\label{equ:z-norm,lv1,H1,fr}\\
\n{\lra{v}^{2}e^{\lra{v}^{2}}f_{\mathrm{r}}}_{L_{t}^{\infty}(0,T_{*};L_{v}^{1}L_{x}^{2})}\lesssim & N^{-1},\label{equ:z-norm,lv1,Lx2,fr}\\
M^{-1}\n{\lra{v}^{2}e^{\lra{v}^{2}}\nabla_{x}f_{\mathrm{r}}}_{L_{t}^{\infty}(0,T_{*};L_{v}^{1}L_{x}^{2})}
\lesssim& N^{-1}.\label{equ:z-norm,lv1,Lx2,H1,fr}
\end{align}

\begin{lemma}\label{lemma:Z-norm estimate,fr,mu}
For $0\leq  t_{0}\leq t \leq T_{*}$, we have
\begin{align}\label{equ:Z norm,Q-+,mu,fr}
\bbn{e^{\lra{v}^{2}} \int_{t_{0}}^{t}S_{\nu}(t-\tau)\lrc{Q^{\pm}(f_{\mathrm{r}},\mu)\pm Q^{\pm}(\mu,f_{\mathrm{r}})}(\tau)d\tau}_{Z}\lesssim M^{-\delta}.
\end{align}
\end{lemma}
\begin{proof}
Applying the weighted $Z$-norm collision estimates \eqref{equ:Z-norm,estimates,mu,f} in Lemma \ref{lemma:weighted,Z-norm,estimates}, we obtain
\begin{align*}
&\n{e^{\lra{v}^{2}}Q^{+}(\mu,f_{\mathrm{r}})}_{Z}+
\n{e^{\lra{v}^{2}}Q^{+}(f_{\mathrm{r}},\mu)}_{Z}+\n{e^{\lra{v}^{2}}Q^{-}(\mu,f_{\mathrm{r}})}_{Z}\\
\lesssim& M^{s_{0}}\n{\lra{v}^{2}e^{\lra{v}^{2}}f_{\mathrm{r}}}_{L_{x}^{2}L_{v}^{1}}
+M^{-1}\n{\lra{v}^{2}e^{\lra{v}^{2}}\nabla_{x}f_{\mathrm{r}}}_{L_{x}^{2}L_{v}^{1}}\notag\\
&+\n{\lra{v}^{2}e^{\lra{v}^{2}}f_{\mathrm{r}}}_{L_{v}^{1}L_{x}^{\infty}}+
M^{-1}\n{\lra{v}^{2}e^{\lra{v}^{2}}\nabla_{x}f_{\mathrm{r}}}_{L_{v}^{1}L_{x}^{\infty}},\notag
\end{align*}
which, together with estimates \eqref{equ:z-norm,lv1,fr}--\eqref{equ:z-norm,lv1,Lx2,H1,fr}, gives
\begin{align}\label{equ:Q+,mu,fr,Z-norm}
\n{e^{\lra{v}^{2}}Q^{+}(\mu,f_{\mathrm{r}})}_{Z}+
\n{e^{\lra{v}^{2}}Q^{+}(f_{\mathrm{r}},\mu)}_{Z}+\n{e^{\lra{v}^{2}}Q^{-}(\mu,f_{\mathrm{r}})}_{Z}\lesssim N^{-1}.
\end{align}

For the term $Q^{-}(f_{\mathrm{r}},\mu)=\nu f_{\mathrm{r}}$, it follows from \eqref{equ:z-norm,lv1,fr} that
\begin{align}\label{equ:Q-,mu,fr,Z-norm}
\n{e^{\lra{v}^{2}}Q^{-}(f_{\mathrm{r}},\mu)}_{Z}\lesssim \n{\lra{v}e^{\lra{v}^{2}} f_{\mathrm{r}}}_{Z}\lesssim \n{f_{\mathrm{r}}}_{Z}\lesssim 1.
\end{align}
Consequently, combining \eqref{equ:Q+,mu,fr,Z-norm} and \eqref{equ:Q-,mu,fr,Z-norm}, we arrive at
\begin{align*}
&\bbn{e^{\lra{v}^{2}} \int_{t_{0}}^{t}S_{\nu}(t-\tau)\lrc{Q^{\pm}(f_{\mathrm{r}},\mu)\pm Q^{\pm}(\mu,f_{\mathrm{r}})}(\tau)d\tau}_{Z}\\
\lesssim& T_{*}\bn{e^{\lra{v}^{2}}Q^{\pm}(f_{\mathrm{r}},\mu)}_{L_{t}^{\infty}(0,T_{*};Z)}+
\bn{e^{\lra{v}^{2}}Q^{\pm}(\mu,f_{\mathrm{r}})}_{L_{t}^{\infty}(0,T_{*};Z)}\\
\lesssim& T_{*}\lesssim M^{s-1}\ln \ln M\lesssim M^{-\delta},
\end{align*}
which is sufficient for our purpose, upon choosing $\delta<\frac{1-s}{2}$.
\end{proof}

\subsubsection{Analysis of $Q^{-}(\mu,f_{\mathrm{b}})$}\label{section:Analysis,Q-,mu,fb}
\begin{lemma}\label{lemma:Z-norm estimate,Q-,fb,mu}
For $0\leq  t_{0}\leq t \leq T_{*}$, we have
	\begin{align}\label{equ:Z norm,Q-,mu,fb}
\bbn{ e^{\lra{v}^{2}}\int_{t_{0}}^{t}S_{\nu}(t-\tau)Q^{-}(\mu,f_{\mathrm{b}})(\tau)\,d\tau} _{Z}\lesssim M^{-\delta}.
	\end{align}
\end{lemma}
\begin{proof}
In view of the definition \eqref{equ:Z-norm,error term estimate} of the
$Z$-norm, it suffices to establish the estimates for the two norms
$M^{s_{0}}\Vert e^{\lra{v}^{2}} \cdot \Vert_{L_{v}^{2,1}L_{x}^{2}}$ and $\Vert e^{\lra{v}^{2}} \cdot \Vert _{L_{v}^{1,1}L_{x}^{\infty}}$. The remaining estimates involving spatial $x$-derivatives, namely $M^{s_{0}-1}\n{e^{\lra{v}^{2}}\nabla_{x}\cdot}_{L_{v}^{2,1}L_{x}^{2}}$ and $M^{-1}\n{e^{\lra{v}^{2}}\nabla_{x}\cdot}_{L_{v}^{1,1}L_{x}^{\infty}}$, follow from an analogous argument.
Indeed, each spatial derivative $\nabla_{x}$ contributes an additional factor of $M$, which is already compensated by the factor
$M^{-1}$ present in the corresponding norms.

\vspace{1em}
\textbf{The $M^{s_{0}}\Vert e^{\lra{v}^{2}} \cdot \Vert_{L_{v}^{2,1}L_{x}^{2}}$ estimate.}

Applying the collision estimate \eqref{equ:Q-,f,mu,Lv2Lx2} in Lemma \ref{lemma:weighted,Z-norm,estimates}, we obtain
\begin{align}\label{equ:Q-,fb,mu,Lv2Lx2}
&M^{s_{0}}\bbn{ e^{\lra{v}^{2}}\int_{t_{0}}^{t}S_{\nu}(t-\tau)Q^{-}(\mu,f_{\mathrm{b}})(\tau)\,d\tau}_{L_{v}^{2,1}L_{x}^{2}}\\
\lesssim& T_{*}M^{s_{0}}\n{e^{\lra{v}^{2}}Q^{-}(\mu,f_{\mathrm{b}})}_{L_{t}^{\infty }(0,T_{*};L_{v}^{2,1}L_{x}^{2})}\notag\\
\lesssim& T_{*}M^{s_{0}}\n{\lra{v}^{2}e^{\lra{v}^{2}}f_{\mathrm{b}}}_{L_{t}^{\infty }(0,T_{*};L_{x}^{2}L_{v}^{1})}.\notag
\end{align}
Then, applying the Sobolev norm estimate \eqref{equ:sobolev norm,fb,L2L1} for $f_{\mathrm{b}}$ from Lemma \ref{lemma:Hs,bounds on fb}, and substituting $T_{*}=M^{s-1}(\ln\ln M)$, we arrive at
\begin{align*}
M^{s_{0}}\bbn{ e^{\lra{v}^{2}}\int_{t_{0}}^{t}S_{\nu}(t-\tau)Q^{+}(f_{\mathrm{b}},\mu)(\tau)\,d\tau}_{L_{v}^{2}L_{x}^{2}}\lesssim& T_{*}M^{s_{0}}M^{-\frac{1}{2}-s}
\lesssim M^{-\frac{3}{2}+s_{0}}(\ln\ln M),
\end{align*}
which meets our requirement with $s_{0}<1$.

\vspace{1em}
\textbf{The $\Vert e^{\lra{v}^{2}} \cdot \Vert _{L_{v}^{1,1}L_{x}^{\infty}}$ estimate.}

For notational convenience, we define
\begin{equation}\label{equ:D-,mu,fb}
D^{-}(\mu,f_{\mathrm{b}})=:\int_{t_{0}}^{t}S_{\nu}(t-\tau)Q^{-}(\mu,f_{\mathrm{b}})(\tau)d\tau=
\int_{t_{0}}^{t}S_{\nu}(t-\tau)(\mu A[f_{\mathrm{b}}](\tau))d\tau.
\end{equation}
In view of the pointwise estimate \eqref{equ:upper bound,Afb} for $A[f_{\mathrm{b}}]$, we have
\begin{align}\label{equ:Afb,upper bound}
A[f_{\mathrm{b}}](t)\lesssim&  M^{1-s}\lrs{\frac{1}{M|x|+1}}^{2}.
\end{align}
Substituting \eqref{equ:Afb,upper bound} into \eqref{equ:D-,mu,fb}, we obtain
\begin{align*}
D^{-}(\mu,f_{\mathrm{b}})
\lesssim &\int_{t_{0}}^{t} M^{1-s} \mu(v)  \lrs{\frac{1}{
M|x-v(t-\tau)|+1}}^{2}d\tau.\notag\\
\lesssim& M^{1-s}\mu(v)
\int_{0}^{t}  \lrs{\frac{1}{
M|x-v(t-\tau)|+1}}^{2}d\tau\\
\lesssim & M^{1-s}\frac{\mu(v)}{M|v|}\lesssim M^{-s}\frac{\mu(v)}{|v|}.
\end{align*}
Therefore, noting that $\mu(v)=e^{-2|v|^{2}}$, we conclude that
\begin{align}\label{equ:Q-,fb,fb,L1,final}
\n{e^{\lra{v}^{2}}D^{-}(\mu,f_{\mathrm{b}})}_{L_{v}^{1,1}L_{x}^{\infty}}\lesssim& M^{-s} \bbn{\lra{v}e^{\lra{v}^{2}}\frac{\mu(v)}{|v|}}_{L_{v}^{1}}
	\lesssim M^{-\delta},
\end{align}
provided that $\delta<s$.
\end{proof}

\subsubsection{Analysis of $Q^{+}(f_{\mathrm{b}},\mu)$, and $Q^{+}(\mu,f_{\mathrm{b}})$}
\label{section:Analysis,Q+,mu,fb}
\begin{lemma}\label{lemma:Z-norm estimate,Q+,fb,mu}
For $0\leq  t_{0}\leq t \leq T_{*}$, we have
\begin{align}\label{equ:Z norm,Q+,mu,fb}	
\bbn{e^{\lra{v}^{2}}\int_{t_{0}}^{t}S_{\nu}(t-\tau)
\lrc{Q^{+}(f_{\mathrm{b}},\mu)+Q^{+}(\mu,f_{\mathrm{b}})}(\tau)d\tau}_{Z}\lesssim M^{-\delta}.
	\end{align}
\end{lemma}
\begin{proof}
By the symmetry of the collision estimates for the gain term $Q^{+}$, we only need to deal with $Q^{+}(f_{\mathrm{b}},\mu)$, while the treatment of $Q^{+}(\mu,f_{\mathrm{b}})$ is analogous.
As in the case of $Q^{-}(\mu,f_{\mathrm{b}})$,
we restrict our attention to the estimates for
$M^{s_{0}}\Vert e^{\lra{v}^{2}} \cdot \Vert_{L_{v}^{2,1}L_{x}^{2}}$ and $\Vert e^{\lra{v}^{2}} \cdot \Vert _{L_{v}^{1,1}L_{x}^{\infty}}$.
The estimates involving spatial derivatives, i.e. $M^{s_{0}-1}\n{e^{\lra{v}^{2}}\nabla_{x}\cdot}_{L_{v}^{2,1}L_{x}^{2}}$ and $M^{-1}\n{e^{\lra{v}^{2}}\nabla_{x}\cdot}_{L_{v}^{1,1}L_{x}^{\infty}}$, follow from a similar way, since each derivative $\nabla_{x}$ introduces an additional factor of
$M$, which is precisely compensated by the factor
$M^{-1}$ in the corresponding norms.

\vspace{1em}
\textbf{The $M^{s_{0}}\Vert e^{\lra{v}^{2}} \cdot \Vert_{L_{v}^{2,1}L_{x}^{2}}$ estimate.}

By the collision estimate \eqref{equ:Q+,sym,f,mu,Lv2Lx2} in Lemma \ref{lemma:weighted,Z-norm,estimates}, we obtain
\begin{align}\label{equ:equ:Q+,fb,fb,pointwise,first}
	&M^{s_{0}}\bbn{ e^{\lra{v}^{2}}\int_{t_{0}}^{t}S_{\nu}(t-\tau)Q^{+}(f_{\mathrm{b}},\mu)(\tau)\,d\tau}_{L_{v}^{2,1}L_{x}^{2}}\\
\lesssim& T_{*}M^{s_{0}}\n{e^{\lra{v}^{2}}Q^{+}(f_{\mathrm{b}},\mu)}_{L_{t}^{\infty }(0,T_{*};L_{v}^{2,1}L_{x}^{2})}\notag\\
\lesssim& T_{*}M^{s_{0}}\n{\lra{v}^{2}e^{\lra{v}^{2}}f_{\mathrm{b}}}_{L_{t}^{\infty }(0,T_{*};L_{x}^{2}L_{v}^{1})}.\notag
\end{align}
Applying Sobolev norm estimate \eqref{equ:sobolev norm,fb,L2L1} for $f_{\mathrm{b}}$ from Lemma \ref{lemma:Hs,bounds on fb} and taking $T_{*}=M^{s-1}(\ln\ln M)$, we arrive at
\begin{align*}
M^{s_{0}}\bbn{ e^{\lra{v}^{2}}\int_{t_{0}}^{t}S_{\nu}(t-\tau)Q^{+}(f_{\mathrm{b}},\mu)(\tau)\,d\tau}_{L_{v}^{2}L_{x}^{2}}\lesssim& T_{*}M^{s_{0}}M^{-\frac{1}{2}-s}
\lesssim M^{-\frac{3}{2}+s_{0}}(\ln\ln M),
\end{align*}
which is enough for our goal since $s_{0}<1$.

\vspace{1em}
\textbf{The $\Vert e^{\lra{v}^{2}} \cdot \Vert _{L_{v}^{1,1}L_{x}^{\infty}}$ estimate.}

We recall that $f_{\mathrm{b}}$ takes the form
\begin{align}\label{equ:fb,D+}
f_{\mathrm{b}}(t,x,v)=&M^{1-s}\sum_{i,j}K_{i,j}(x-vt)I_{i,j}(v).
\end{align}
For brevity, we introduce the notation
\begin{align*}
D^{+}(f_{\mathrm{b}},\mu)=&\int_{t_{0}}^{t}S_{\nu}(t-\tau)Q^{+}(f_{\mathrm{b}},\mu)(\tau)d\tau \\
=&M^{1-s}\sum_{i,j}\int_{t_{0}}^{t}S_{\nu}(t-\tau)Q^{+}(K_{i,j}(x- v\tau )I_{i,j}(v),\mu(v))d\tau\\
= &M^{1-s}\sum_{i,j}\int_{
\mathbb{S}^{2}} \int_{\R^{3}} |(u-v)\cdot \omega|
H_{i,j}(t_{0},t,x,\omega,u^{*},v^{*}) dud\omega,\notag
\end{align*}
where
\begin{align*}
H_{i,j}(t_{0},t,x,\omega,u^{*},v^{*})=:\int_{t_{0}}^{t} e^{-\nu(v)(t-\tau)} K_{i,j}(x-v(t-\tau)-v^{*}\tau)I_{i,j}(v^{*}) \mu(u^{*})d\tau.
\end{align*}

Using the inequalities that
\begin{align*}
\lra{v}e^{\lra{v}^{2}}\lesssim e^{\frac{3}{2}\lrs{|u^{*}|^{2}+|v^{*}|^{2}}},
\quad |(u-v)\cdot \omega|\lesssim |u-v|,
\end{align*}
we have
\begin{align}
&\n{\lra{v}e^{\lra{v}^{2}}D^{+}(f_{\mathrm{b}},\mu)}_{L_{v}^{1}L_{x}^{\infty}}\\
\lesssim&  M^{1-s}\sum_{i,j}\int_{
\mathbb{S}^{2}} \int_{\R^{3}\times \R^{3}} \bbn{|u-v|
e^{\frac{3}{2}\lrs{|u^{*}|^{2}+|v^{*}|^{2}}}H_{i,j}(t_{0},t,x,\omega,u^{*},v^{*})}_{L_{x}^{\infty}} dvdud\omega\notag\\
=&M^{1-s}\sum_{i,j}\int_{
\mathbb{S}^{2}} \int_{\R^{3}\times \R^{3}} \bbn{|u-v|
e^{\frac{3}{2}\lrs{|u|^{2}+|v|^{2}}}H_{i,j}(t_{0},t,x,\omega,u,v)}_{L_{x}^{\infty}} dvdud\omega,\notag
\end{align}
where in the last inequality we have performed the collision transform $(u^{*},v^{*})\mapsto (u,v)$ and used the identity
$|u-v|=|u^{*}-v^{*}|$.

Furthermore,
using that
\begin{align*}
e^{-\nu(v^{*})(t-\tau)} \lesssim 1,\quad  \mu(u)=e^{-2|u|^{2}},\quad K_{i,j}(x)=\chi (MP_{e_{i,j}}^{\perp }x)\chi (P_{e_{i,j}}x)\lesssim\chi (MP_{e_{i,j}}^{\perp}x),
\end{align*}
together with the identity
\begin{equation*}
\begin{aligned}
&x-v^{*}(t-\tau)-v\tau =x-v^{*}t-P_{\omega}(v-u)\tau,\quad \text{$P_{\omega}$: the projection in the direction of $\omega$},
\end{aligned}
\end{equation*}
we deduce
\begin{align*}
&|u-v|e^{\frac{3}{2}\lrs{|u|^{2}+|v|^{2}}}H_{i,j}(t_{0},t,x,\omega,u,v)\\
=& \int_{t_{0}}^{t} |u-v|e^{\frac{3}{2}\lrs{|u|^{2}+|v|^{2}}}e^{-\nu(v^{*})(t-\tau)} K_{i,j}(x-v^{*}(t-\tau)-v\tau)I_{i,j}(v) \mu(u)d\tau\notag\\
\lesssim&I_{i,j}(v)e^{-\frac{|u|^{2}}{2}}|u-v|E_{i,j}(t,x,\omega,u,v),\notag
\end{align*}
where
\begin{align*}
E_{i,j}(t,x,\omega,u,v)
=:&\int_{0}^{T_{*}}\chi \left( MP_{e_{i,j}}^{\perp }(x-v^{*}t-P_{\omega}(v-u)\tau)\right) d\tau.
\end{align*}

Consequently, we arrive at
\begin{align}
&\n{\lra{v}e^{\lra{v}^{2}}D^{+}(f_{\mathrm{b}},\mu)}_{L_{v}^{1}L_{x}^{\infty}}\label{equ:estimate,D+,Eij}\\
\lesssim&
M^{1-s}\sum_{i,j} \int_{\R^{3}\times \R^{3}} I_{i,j}(v)e^{-\frac{|u|^{2}}{2}}|u-v|\lrc{ \int_{\mathbb{S}^{2}}\n{E_{i,j}(t,x,\omega,u,v)}_{L_{x}^{\infty}}d\omega} dudv.\notag
\end{align}

We now turn to the analysis of $E_{i,j}(t,x,\omega,u,v)$.
First, it admits the trivial upper bound
\begin{align}\label{equ:E,trivial bound}
E_{i,j}(t,x,\omega,u,v)\leq T_{*}\leq  1.
\end{align}
However, this trivial bound alone is insufficient to produce the desired decay, and a more refined estimate is therefore required.

To proceed, we use the following time-integration estimate
\begin{align}\label{equ:chi,inequality}
\int_{\R}\chi(\overrightarrow{n}\tau+\overrightarrow{m})d\tau=
\frac{1}{|\overrightarrow{n}|}
\int_{\R}\chi\lrs{\frac{\overrightarrow{n}}{|\overrightarrow{n}|}\tau+\overrightarrow{m}}d\tau
\lesssim \frac{1}{|\overrightarrow{n}|},
\end{align}
which effectively reduces the bound to a geometric constraint involving the projection directions.
By this time-integration estimate \eqref{equ:chi,inequality}, we arrive at
\begin{align}\label{equ:E,nontrivial bound}
E_{i,j}(t,x,\omega,u,v)=&\int_{0}^{T_{*}}\chi \left( MP_{e_{i,j}}^{\perp }(x-v^{*}t-P_{\omega}(v-u)\tau)\right)d\tau\\
\lesssim& \frac{1}{M|P_{e_{i,j}}^{\perp}P_{\omega
}(v-u)|}\notag\\
\lesssim&  \frac{1}{M\sin \phi_{i,j}(\omega)}
\frac{1}{|P_{\omega
}(v-u)|}\notag\\
\lesssim & \frac{1}{M\sin \phi_{i,j}(\omega)} \frac{1}{\cos \theta(u-v,\omega)}\frac{1}{|u-v|},\notag
\end{align}
where $\phi_{i,j}(\omega)\in [0,\pi)$ is the angle between $\omega$ and $e_{i,j}$, and $\theta(u-v,\omega)\in [0,\pi)$ be the angle between $\omega$ and $v-u$.

With these estimates for $E_{i,j}$ in hand, we proceed to the angular integration. We partition the sphere $\mathbb{S}^{2}$ through an angular decomposition and define
\begin{align*}
\Omega=\lr{\omega\in \mathbb{S}^{2}:\sin \phi_{i,j}(\omega)\leq M^{-1}}\bigcup \lr{\omega\in \mathbb{S}^{2}:\abs{\cos \theta(u-v,\omega)}\leq M^{-\frac{1}{2}}},
\end{align*}
with $\Omega^{c}$ denoting the complementary set of $\Omega$. Then, the integral can be split into two parts
\begin{align*}
\int_{\mathbb{S}^{2}} \n{E_{i,j}(t,x,\omega,u,v)}_{L_{x}^{\infty}}d\omega
\leq A_{1}+A_{2},
\end{align*}
where
\begin{align*}
A_{1}=&\int_{\Omega}\n{E_{i,j}(t,x,\omega,u,v)}_{L_{x}^{\infty}}d\omega,\\
 A_{2}=&\int_{\Omega^{c}}\n{E_{i,j}(t,x,\omega,u,v)}_{L_{x}^{\infty}}d\omega.
\end{align*}

 For $A_{1}$, applying the trivial bound \eqref{equ:E,trivial bound} on the set $\Omega$ and utilizing spherical coordinates, we obtain
\begin{align}\label{equ:Eij,A1,estimate}
&\int_{\Omega}\n{E_{i,j}(t,x,\omega,u,v)}_{L_{x}^{\infty}}d\omega\\
 \leq& \int_{\lr{\omega\in \mathbb{S}^{2}:\sin \phi_{i,j}(\omega)\leq M^{-1}}}1d\omega+\int_{\lr{\omega\in \mathbb{S}^{2}:\abs{\cos \theta(u-v,\omega)}\leq M^{-\frac{1}{2}}}}1d\omega\notag\\
\lesssim& \int_{0}^{\pi}1_{\lr{\sin \phi\leq M^{-1}}}\sin \phi d\phi+
\int_{0}^{\pi}1_{\lr{|\cos\theta|\leq M^{-\frac{1}{2}}}}\sin \theta d\theta\notag\\
\lesssim& M^{-2}+M^{-\frac{1}{2}}\lesssim M^{-\frac{1}{2}}.\notag
\end{align}

For $A_{2}$, applying the nontrivial upper bound \eqref{equ:E,nontrivial bound}, we get
\begin{align}
&
\int_{\Omega^{c}}\n{E_{i,j}(t,x,\omega,u,v)}_{L_{x}^{\infty}}d\omega\notag\\
\leq&\frac{1}{M|u-v|}\int_{\Omega^{c}} \frac{1}{\sin \phi_{i,j}(\omega)\cos \theta(u-v,\omega)}d\omega\notag\\
\lesssim & \frac{1}{M |u-v|}\lrc{\int_{\Omega^{c}}\frac{1}{(\sin \phi_{i,j}(\omega))^{2}}d\omega+\int_{\Omega^{c}}\frac{1}{(\cos \theta(u-v,\omega))^{2}}d\omega}.\label{equ:estimate,Eij,Omegac,two terms}
\end{align}
For the last two terms on the right hand side of \eqref{equ:estimate,Eij,Omegac,two terms}, by surface integral formula, we derive
\begin{align}\label{equ:Eij,A2,estimate,sin}
\int_{\Omega^{c}}\frac{1}{(\sin \phi_{i,j}(\omega))^{2}}d\omega\lesssim &
\int_{0}^{\pi}1_{\lr{\sin \phi\geq M^{-1}}}\frac{\sin \phi}{(\sin \phi)^{2}}d\phi
\lesssim\int_{M^{-1}}^{1}\frac{1}{\phi}d\phi \lesssim \ln M,
\end{align}
\begin{align}\label{equ:Eij,A2,estimate,cos}
\int_{\Omega^{c}}\frac{1}{(\cos \theta(u-v,\omega))^{2}}d\omega\lesssim & \int_{0}^{\pi}1_{\lr{|\cos\theta|\geq M^{-\frac{1}{2}}}}\frac{\sin \theta}{(\cos \theta)^{2}} d\theta
\lesssim \int_{M^{-\frac{1}{2}}}^{1}\frac{1}{\la^{2}}d\la\lesssim M^{\frac{1}{2}}.
\end{align}

Putting together \eqref{equ:Eij,A1,estimate}, \eqref{equ:estimate,Eij,Omegac,two terms}, \eqref{equ:Eij,A2,estimate,sin}, \eqref{equ:Eij,A2,estimate,cos}, we arrive at
\begin{align}\label{equ:Sij,estimate,D+}
\int_{\mathbb{S}^{2}}\n{E_{i,j}(t,x,\omega,u,v)}_{L_{x}^{\infty}}d\omega \lesssim \frac{1}{M^{\frac{1}{2}}|u-v|}.
\end{align}
Substituting this bound \eqref{equ:Sij,estimate,D+} back into the estimate \eqref{equ:estimate,D+,Eij} for $D^{+}(f_{\mathrm{b}},\mu)$, we have
\begin{align*}
&\n{e^{\lra{v}^{2}}D^{+}(f_{\mathrm{b}},\mu)(f_{\mathrm{b}},\mu)}_{L_{v}^{1,1}L_{x}^{\infty }}\\
\lesssim&  M^{1-s}\sum_{i,j}\int_{\R^{3}\times \R^{3}}
I_{i,j}(v) e^{-\frac{|u|^{2}}{2}}|u-v|\frac{1}{M^{\frac{1}{2}}|u-v|}dudv\\
\lesssim&   M^{1-s}M^{-\frac{1}{2}}\sum_{i,j}\int_{\R^{3}\times \R^{3}}I_{i,j}(v) e^{-\frac{|u|^{2}}{2}}
dudv\\
\lesssim& M^{\frac{1}{2}-s},
\end{align*}
where in the last inequality we have used the summation estimate \eqref{equ:Iij,sum} for $I_{ij}$ that
\begin{align*}
\sum_{i,j}I_{i,j}(v)\leq& 1_{\lr{\frac{3}{4}\leq|v|\leq \frac{5}{4}}}(v).
\end{align*}
Therefore, we complete the proof of the desired estimate \eqref{equ:Z norm,Q+,mu,fb}.
\end{proof}

\subsection{$Z$-norm Estimates for Correction Terms}\label{section:Z-norm Estimates for Correction Terms}
In this section, the main goal is to derive uniform-in-time $Z$-norm estimates for the correction terms $f_{\mathrm{c}}(t)$. In addition, we construct a second solution
$g(t)$, which is stable near the initial datum $f_{\mathrm{r}}(0)$.

We begin by providing a space-time bilinear estimate.
\begin{proposition}\label{lemma:Z,norm,bilinear,estimate}
 Let $0\leq t_{0}\leq t\leq \frac{1}{\kappa}$. Then the following estimates in time-integral form hold
\begin{align}\label{equ:Z,norm,bilinear,estimate}
&\bbn{e^{(1-\kappa t)\lra{v}^{2}}\int_{t_{0}}^{t}S_{\nu}(t-\tau)Q^{\pm}(f_{1},f_{2})(\tau)d\tau}_{Z}\\
\lesssim & \lrs{\frac{t-t_{0}}{\kappa}}^{\frac{1}{2}}\n{e^{(1-\kappa \tau)\lra{v}^{2}}f_{1}}_{L_{\tau}^{\infty}(t_{0},t;Z)}
\n{e^{(1-\kappa \tau)\lra{v}^{2}}f_{2}}_{L_{\tau}^{\infty}(t_{0},t;Z)},\notag
\end{align}
\begin{align}\label{equ:Z,norm,bilinear,estimate,Kmu,f}
\bbn{e^{(1-\kappa t)\lra{v}^{2}}\int_{t_{0}}^{t}S_{\nu}(t-\tau)K_{\mu}f(\tau)d\tau}_{Z}
\lesssim & \lrs{\frac{t-t_{0}}{\kappa}}^{\frac{1}{2}}\n{e^{(1-\kappa \tau)\lra{v}^{2}}f}_{L_{\tau}^{\infty}(t_{0},t;Z)}.
\end{align}
\end{proposition}
\begin{proof}
The space-time bilinear estimate \eqref{equ:Z,norm,bilinear,estimate} in the Gaussian-weighted
$Z$-norm has been established for the free transport $S(t)$ in \cite[Proposition 5.1]{CGSZ26}. We note that the same proof applies to the damped case $S_{\nu}(t)$, since the additional damped factor $e^{-\nu t}$ can be discarded directly. Therefore, it remains only to prove \eqref{equ:Z,norm,bilinear,estimate,Kmu,f}.

We first derive a pointwise bound in $L_{x}^{p}$ norm. Discarding the damped transport operator $S_{\nu}(t-\tau)$, we obtain
\begin{align}\label{equ:space-time,pointwise,K,mu}
&\bbn{e^{(1-\kappa t)\lra{v}^{2}}
\int_{t_{0}}^{t}S_{\nu}(t-\tau)K_{\mu}f(\tau)d\tau}_{L_{x}^{p}}\\
\leq &\int_{t_{0}}^{t}e^{(1-\kappa t)\lra{v}^{2}}\n{K_{\mu}f(\tau)}_{L_{x}^{p}}d\tau\notag\\
= &\int_{t_{0}}^{t}e^{\kappa (\tau-t)\lra{v}^{2}}e^{(1-\kappa \tau)\lra{v}^{2}}\n{K_{\mu}f(\tau)}_{L_{x}^{p}}d\tau\notag\\
\lesssim &
\int_{t_{0}}^{t}\frac{1}{\sqrt{\ka(t-\tau)}}\n{\lra{v}^{-1}e^{(1-\kappa \tau)\lra{v}^{2}}K_{\mu}f(\tau)}_{L_{x}^{p}}d\tau,\notag
\end{align}
where in the last inequality we have used that $e^{-a\lra{v}^{2}}\lesssim \frac{1}{\sqrt{a}\lra{v}} $.

We now turn to the estimates in the $Z$-norm. Following the proof of Lemma \ref{lemma:weighted,Z-norm,estimates},
we have
\begin{align}\label{equ:Kmu,Z,v}
\n{\lra{v}^{-1}e^{(1-\kappa \tau)\lra{v}^{2}}K_{\mu}f(\tau)}_{Z}\lesssim \n{e^{(1-\kappa \tau)\lra{v}^{2}}f(\tau)}_{Z}.
\end{align}
Combining \eqref{equ:space-time,pointwise,K,mu} and \eqref{equ:Kmu,Z,v}, we arrive at
\begin{align*}
\bbn{e^{(1-\kappa t)\lra{v}^{2}}
\int_{t_{0}}^{t}S_{\nu}(t-\tau)K_{\mu}f(\tau)d\tau}_{Z}
\lesssim&
\int_{t_{0}}^{t}\frac{1}{\sqrt{\ka(t-\tau)}}\n{\lra{v}^{-1}e^{(1-\kappa \tau)\lra{v}^{2}}K_{\mu}f(\tau)}_{Z}d\tau\\
\lesssim& \lrs{\frac{t-t_{0}}{\kappa}}^{\frac{1}{2}}\n{e^{(1-\kappa \tau)\lra{v}^{2}}f}_{L_{\tau}^{\infty}(t_{0},t;Z)},
\end{align*}
which completes the proof of \eqref{equ:Z,norm,bilinear,estimate,Kmu,f}.
\end{proof}

Next, we establish uniform-in-time $Z$-norm estimates for the correction term $f_{\mathrm{c}}(t)$,
which is governed by the following equation
\begin{equation}\label{equ:correction term,fc,stability}
\left\{
\begin{aligned} \pa_{t}f_{\mathrm{c}} + v\cdot \nabla_{x} f_{\mathrm{c}}+\nu f_{\mathrm{c}}
=&\pm Q^\pm(f_{\mathrm{c}},f_{ \mathrm{a}}) \pm
Q^\pm(f_{\mathrm{a}},f_{\mathrm{c}}) \pm Q^\pm(f_{\mathrm{c}},f_{\mathrm{c}
})+ K_{\mu}f_{\mathrm{c}}-F_{\text{err}}^{(1)}-F_{\text{err}}^{(2)},\\
F_{\mathrm{err}}^{(1)}=&\pa_{t}f_{\mathrm{a}}+v\cdot \nabla_{x}f_{\mathrm{a}}+Q^{-}(f_{\mathrm{a}},f_{\mathrm{a}})-Q^{+}(f_{\mathrm{a}},f_{\mathrm{a}})+\nu f_{\mathrm{b}},\\
F_{\mathrm{err}}^{(2)}
=&\mp Q^{\pm}(f_{\mathrm{r}},\mu)\pm Q^{\pm}(\mu,f_{\mathrm{r}})-Q^{+}(f_{\mathrm{b}},\mu)-Q^{+}(\mu,f_{\mathrm{b}})+Q^{-}(\mu,f_{\mathrm{b}}),\\
f_{\mathrm{c}}(0)=&0.
 \end{aligned}
 \right.
\end{equation}
In addition, in order to compare two distinct solutions, we also construct a second solution $g(t)=g_{\mathrm{a}}(t)+g_{\mathrm{c}}(t)$, whose approximate part is taken to be the stationary state $f_{\mathrm{r}}(0)$, namely,
$$g_{\mathrm{a}}(t)\equiv f_{\mathrm{r}}(0).$$
Correspondingly, the correction term $g_{\mathrm{c}}(t)$ satisfies
\begin{equation}\label{equ:correction term,gc,stability}
\left\{
\begin{aligned} \pa_{t}g_{\mathrm{c}} + v\cdot \nabla_{x} g_{\mathrm{c}}+\nu g_{\mathrm{c}}
=&\pm Q^\pm(g_{\mathrm{c}},g_{ \mathrm{a}}) \pm
Q^\pm(g_{\mathrm{a}},g_{\mathrm{c}}) \pm Q^\pm(g_{\mathrm{c}},g_{\mathrm{c}
})+ K_{\mu}g_{\mathrm{c}}-G_{\text{err}},\\
G_{\mathrm{err}}=&\pa_{t}g_{\mathrm{a}}+v\cdot \nabla_{x}g_{\mathrm{a}}+Q^{-}(g_{\mathrm{a}},g_{\mathrm{a}})-Q^{+}(g_{\mathrm{a}},g_{\mathrm{a}})+\nu g_{\mathrm{a}}-K_{\mu}g_{\mathrm{a}},\\
g_{\mathrm{c}}(0)=&0.
 \end{aligned}
 \right.
\end{equation}

This choice of $g_{\mathrm{a}}(t)$ ensures that the two exact solutions $f(t)$ and $g(t)$ are initially close in the strong weighted norm.
In addition, based on the uniform estimates for the correction terms $f_{\mathrm{c}}(t)$ and $g_{\mathrm{c}}(t)$, the norm deflation property previously established for $f_{\mathrm{a}}(t)$, ensures that the two solutions become fully separated at the critical time $T_{*}$.

To obtain the uniform estimates for $f_{\mathrm{c}}(t)$ and $g_{\mathrm{c}}(t)$ on $[0,T_{*}]$,
we carry out an iterative scheme based on the space-time collision estimates in the Gaussian-weighted $Z$-norm,
established in Proposition \ref{lemma:Z,norm,bilinear,estimate}, the
$Z$-norm estimates for $f_{\mathrm{a}}(t)$ and $f_{\mathrm{r}}(t)$, and the uniform-in-time estimates for the error terms from Section \ref{section:Estimates for Error Terms}.

\begin{proposition}\label{lemma:perturbation}
Let $T_{*}=M^{s-1}(\ln\ln M)$ and $\kappa=\frac{1}{T_{*}}$.
 Then, we have
\begin{align}
\sup_{t\in[0,T_{*}]}\n{e^{(1-\kappa t)\lra{v}^{2}}f_{\mathrm{c}}(t)}_{Z}\lesssim& M^{-\frac{\delta}{2}},  \label{E:fc_bound2}\\
\sup_{t\in[0,T_{*}]}\n{e^{(1-\kappa t)\lra{v}^{2}}g_{\mathrm{c}}(t)}_{Z}\lesssim& M^{-\frac{\delta}{2}}.  \label{E:gc_bound2}
	\end{align}
In particular, we have
\begin{align}
\sup_{t\in[0,T_{*}]}\n{e^{(1-\kappa t)\lra{v}^{2}}f_{\mathrm{c}}(t)}_{L_{v}^{2} H_{x}^{s_{0}}}\lesssim& M^{-\frac{\delta}{2}},  \label{E:fc_bound2,sobolev}\\
\sup_{t\in[0,T_{*}]}\n{e^{(1-\kappa t)\lra{v}^{2}}g_{\mathrm{c}}(t)}_{L_{v}^{2} H_{x}^{s_{0}}}\lesssim& M^{-\frac{\delta}{2}}.  \label{E:gc_bound2,sobolev}
\end{align}
\end{proposition}

\begin{proof}
Following a standard interpolation argument, we have
\begin{align*}
\n{f}_{L_{v}^{2}H_{x}^{s_{0}}}\leq \n{f}_{L_{v}^{2}L_{x}^{2}}^{1-s_{0}}\n{\nabla_{x}f}_{L_{v}^{2}L_{x}^{2}}^{s_{0}}
\leq
\lrs{M^{s_{0}}\n{f}_{L_{v}^{2}L_{x}^{2}}}^{1-s_{0}} \lrs{M^{s_{0}-1}\n{\nabla_{x}f}_{L_{v}^{2}L_{x}^{2}}}^{s_{0}}\leq&\n{f}_{Z}.
\end{align*}
Therefore, we need only consider the $Z$-norm estimates \eqref{E:fc_bound2}--\eqref{E:gc_bound2}.

We first prove the estimate \eqref{E:fc_bound2} for $f_{\mathrm{c}}(t)$. The argument for $g_{\mathrm{c}}(t)$ is analogous and will be summarized at the end.
We partition the time interval $[0,T_{*}]$ as
\begin{equation*}
0=T_{0}<T_{1}<T_{2}<\cdots <T_{n-1}<T_{n}=T_{*}=M^{s-1}\ln\ln M,
\end{equation*}
	where
$$T_{j}= \frac{jM^{s-1}}{\sqrt{\ln M}},\quad n= \sqrt{\ln M} (\ln\ln M).$$
Consequently, each subinterval $I_{j}=[T_{j},T_{j+1}]$ is of length
\begin{equation*}
|I_{j}|=\frac{M^{s-1}}{\sqrt{\ln M}}.
\end{equation*}

	For $t\in I_{j}=[T_{j},T_{j+1}]$, using the Duhamel formula of equation \eqref{equ:correction term,fc,stability}, we have
\begin{align*}
f_{\mathrm{c}}(t)=S_{\nu}(t-T_{j})f_{\mathrm{c}}(T_{j})+
\int_{T_{j}}^{t}S_{\nu}(t-\tau)F(\tau)d\tau,
\end{align*}
where the forcing term
$F$ is defined as
$$F=\pm Q^\pm(f_{\mathrm{c}},f_{ \mathrm{a}}) \pm
Q^\pm(f_{\mathrm{a}},f_{\mathrm{c}}) \pm Q^\pm(f_{\mathrm{c}},f_{\mathrm{c}
})+ K_{\mu}f_{\mathrm{c}}-F_{\text{err}}^{(1)}-F_{\text{err}}^{(2)}.$$

	Applying the
space-time collision estimates \eqref{equ:Z,norm,bilinear,estimate} and \eqref{equ:Z,norm,bilinear,estimate,Kmu,f} from Proposition \ref{lemma:Z,norm,bilinear,estimate}, we obtain
\begin{align}\label{equ:fc,Z-norm,gauss}
		&\n{e^{(1-\kappa t)\lra{v}^{2}}f_{\mathrm{c}}}_{L_{t}^{\infty}(I_{j};Z)}\\
 \leq& \n{e^{(1-\kappa T_{j})\lra{v}^{2}}f_{\mathrm{c}}(T_{j})}_{Z}
+\bbn{e^{(1-\kappa t)\lra{v}^{2}}\int_{T_{j}}^{t}S_{\nu}(t-\tau)F(\tau)d\tau}_{L_{t}^{\infty}(I_{j};Z)}\notag \\
		 \lesssim & \n{e^{(1-\kappa T_{j})\lra{v}^{2}}f_{\mathrm{c}}(T_{j})}_{Z}+ \lrs{\frac{|I_{j}|}{\kappa}}^{\frac{1}{2}}
\n{e^{(1-\kappa t)\lra{v}^{2}}f_{\mathrm{a}}}_{L_{t}^{\infty}(I_{j};Z)}  \n{e^{(1-\kappa t)\lra{v}^{2}}f_{\mathrm{c}}}_{L_{t}^{\infty}(I_{j};Z)}\notag\\
&+\lrs{\frac{|I_{j}|}{\kappa}}^{\frac{1}{2}}\n{e^{(1-\kappa t)\lra{v}^{2}}f_{\mathrm{c}}}_{L_{t}^{\infty}(I_{j};Z)}^{2}
+\lrs{\frac{|I_{j}|}{\kappa}}^{\frac{1}{2}}\n{e^{(1-\kappa t)\lra{v}^{2}}f_{\mathrm{c}}}_{L_{t}^{\infty}(I_{j};Z)}		\notag \\
&+\sum_{i=1,2}\bbn{e^{(1-\kappa t)\lra{v}^{2}}\int_{T_{j}}^{t}S_{\nu}(t-\tau)F_{\mathrm{err}}^{(i)}(\tau)d\tau}_{L_{t}^{\infty}(I_{j};Z)}.\notag
	\end{align}
With the parameter choice $\frac{1}{\kappa}=T_{*}=M^{s-1}\ln\ln M$, the coefficient in \eqref{equ:fc,Z-norm,gauss} is bounded by
\begin{align}\label{equ:Ij,kappa}
\lrs{\frac{|I_{j}|}{\kappa}}^{\frac{1}{2}}\lesssim \frac{M^{s-1}(\ln \ln M)^{\frac{1}{2}}}{(\ln M)^{\frac{1}{4}}} \lesssim \frac{M^{s-1}}{(\ln M)^{\frac{1}{10}}}.
\end{align}

For $f_{\mathrm{a}}(t)$, by Lemmas \ref{lemma:Hs,bounds on fb} and \ref{lemma:bounds on fr}, we provides the $Z$-norm estimate
\begin{align}\label{equ:z-norm estimate for fa,fc,proof}
\n{e^{\lra{v}^{2}}f_{\mathrm{a}}(t)}_{L_{t}^{\infty}(0,T_{*};Z)}\lesssim M^{1-s}.
\end{align}

For the error terms $F_{\mathrm{err}}^{(1)}$ and $F_{\mathrm{err}}^{(2)}$, we employ the uniform-in-time decay estimates \eqref{equ:Ferr_bound,1} and \eqref{equ:Ferr_bound,2} from Proposition \ref{lemma:bounds on ferr}, which gives
\begin{equation}\label{equ:Ferr_bound,fc,proof}
\sum_{i=1,2}\bbn{e^{\lra{v}^{2}}\int_{t_{0}}^{t}S_{\nu}(t-\tau)F_{\mathrm{err}
}^{(i)}(\tau)\,d\tau}_{Z}\lesssim M^{-\delta},\quad 0\leq t_{0}\leq t\leq T_{*}.
\end{equation}

Inserting the estimates \eqref{equ:Ij,kappa}, \eqref{equ:z-norm estimate for fa,fc,proof}, and \eqref{equ:Ferr_bound,fc,proof} into \eqref{equ:fc,Z-norm,gauss}, we obtain
\begin{align*}
		\n{e^{(1-\kappa t)\lra{v}^{2}}f_{\mathrm{c}}}_{L_{t}^{\infty}(I_{j};Z)}\leq& \n{e^{(1-\kappa T_{j})\lra{v}^{2}}f_{\mathrm{c}}(T_{j})}_{Z}
+ \frac{C}{(\ln M)^{\frac{1}{10}}}\n{e^{(1-\kappa t)\lra{v}^{2}}f_{\mathrm{c}}}_{L_{t}^{\infty}(I_{j};Z)}\\
&+ \frac{CM^{s-1}}{(\ln M)^{\frac{1}{10}}}\n{e^{(1-\kappa t)\lra{v}^{2}} f_{\mathrm{c}}}_{L_{t}^{\infty}(I_{j};Z)}^{2}+CM^{-\delta},
\end{align*}
where $C$ denotes an absolute constant. By absorbing the linear term $\n{f_{\mathrm{c}}}_{L_{t}^{\infty}(I_{j};Z)}$ on the right-hand side, we deduce
\begin{equation*}
\n{e^{(1-\kappa t)\lra{v}^{2}}f_{\mathrm{c}}}_{L_{t}^{\infty}(I_{j};Z)}\leq 2\n{e^{(1-\kappa T_{j})\lra{v}^{2}}f_{\mathrm{c}
}(T_{j})}_{Z}+2CM^{-\delta}.
\end{equation*}

We now iterate this estimate from $j=0$ to $j=n-1$. Given the initial condition $f_{\mathrm{c}}(0)=0$, we have for $j=0$
\begin{equation*}
\n{e^{(1-\kappa t)\lra{v}^{2}}f_{\mathrm{c}}}_{L_{t}^{\infty}(I_{j};Z)}\leq 2CM^{-\delta}.
\end{equation*}
Proceeding inductively for $j=1,\ldots $, we obtain
\begin{align*}
\n{e^{(1-\kappa t)\lra{v}^{2}}f_{\mathrm{c}}}_{L_{t}^{\infty}(I_{j};Z)}\leq&(2^{j+1}-1)2CM^{-\delta}.
\end{align*}

Since $j+1=n=\sqrt{\ln M}(\ln\ln M)$, we conclude
\begin{equation*}
\sup_{t\in[0,T_{*}]}\Vert e^{(1-\kappa t)\lra{v}^{2}}f_{\mathrm{c}}(t)\Vert _{Z}\leq \frac{Ce^{\sqrt{\ln M}\ln \ln M}}{M^{\delta}} = \frac{Ce^{\sqrt{\ln M}\ln \ln M}}{e^{\delta \ln M}}\leq
M^{-\frac{\delta}{2}}\ll 1.
\end{equation*}

The proof for $g_{\mathrm{c}}(t)$ follows in an analogous manner. We express $g_{\mathrm{c}}(t)$ via the Duhamel formula
\begin{align*}
g_{\mathrm{c}}(t)=S_{\nu}(t-T_{j})g_{\mathrm{c}}(T_{j})+\int_{T_{j}}^{t}S_{\nu}(t-\tau)G(\tau)d\tau,
\end{align*}
where the forcing term
$G(t)$ is given by
\begin{equation*}
\left\{
\begin{aligned}
G=&\pm Q^\pm(g_{\mathrm{c}},g_{ \mathrm{a}}) \pm
Q^\pm(g_{\mathrm{a}},g_{\mathrm{c}}) \pm Q^\pm(g_{\mathrm{c}},g_{\mathrm{c}
})+ K_{\mu}g_{\mathrm{c}}-G_{\text{err}},\\
G_{\mathrm{err}}=&\pa_{t}g_{\mathrm{a}}+v\cdot \nabla_{x}g_{\mathrm{a}}+Q^{-}(g_{\mathrm{a}},g_{\mathrm{a}})-Q^{+}(g_{\mathrm{a}},g_{\mathrm{a}})+\nu g_{\mathrm{a}}-K_{\mu}g_{\mathrm{a}}.
\end{aligned}
\right.
\end{equation*}

Adopting the iterative scheme employed for $f_{\mathrm{c}}(t)$, it remains to establish the counterparts of the key estimates \eqref{equ:z-norm estimate for fa,fc,proof} for $g_{\mathrm{a}}(t)$ and \eqref{equ:Ferr_bound,fc,proof} for $G_{\mathrm{err}}(t)$. We note that the approximate solution is the stationary state, i.e. $g_{\mathrm{a}}(t)\equiv f_{\mathrm{r}}(0)$,
and hence employ estimate \eqref{lemma:bounds on fr} from Lemma \ref{lemma:bounds on fr} to obtain
\begin{align}\label{equ:z-norm estimate for ga,gc,proof}
\n{e^{\lra{v}^{2}}g_{\mathrm{a}}(t)}_{L_{t}^{\infty}(0,T_{*};Z)}=\n{e^{\lra{v}^{2}}f_{\mathrm{r}}(0)}_{Z}\lesssim 1\lesssim M^{1-s}.
\end{align}
Furthermore, due to this fact that $g_{\mathrm{a}}(t)\equiv f_{\mathrm{r}}(0)$, the corresponding error estimates for $G_{\text{err}}$ are easier to establish than those for
the error terms $F_{\mathrm{err}}^{(1)}$ and $F_{\mathrm{err}}^{(2)}$. Consequently, repeating the analysis of Proposition \ref{lemma:bounds on ferr}, we also have
\begin{align}\label{equ:Ferr_bound,gc,proof}
&\bbn{e^{\lra{v}^{2}}\int_{t_{0}}^{t}S_{\nu}(t-\tau)G_{\mathrm{err}
}(\tau)\,d\tau}_{Z}
\lesssim M^{-\delta}.
\end{align}
With \eqref{equ:z-norm estimate for ga,gc,proof} and \eqref{equ:Ferr_bound,gc,proof} in hand, the same iterative argument applied to $g_{\mathrm{c}}(t)$ leads to the desired estimate \eqref{E:gc_bound2}. This completes the proof.

\end{proof}

\subsection{Proof of Instability}\label{section:Proof of Instability}
We are now ready to present the proof of the instability result in Theorem \ref{thm:main theorem,instability}.
The proof is built upon three key estimates, the Sobolev norm estimates for $f_{\mathrm{b}}(t)$ in Lemma \ref{lemma:Hs,bounds on fb}, Sobolev norm estimates for $f_{\mathrm{r}}(t)$ from Lemma \ref{lemma:bounds on fr}, and uniform-in-time decay estimates for the correction terms $f_{\mathrm{c}}(t)$ and $g_{\mathrm{c}}(t)$ provided in Proposition \ref{lemma:perturbation}.

\begin{proof}[\textbf{Proof of instability in Theorem $\ref{thm:main theorem,instability}$}]
We consider the two solutions constructed in the previous Sections \ref{sec:Approximate Solution and Norm Deflation}--\ref{section:Z-norm Estimates for Correction Terms}, which admit the following decompositions
\begin{equation*}
\left\{
\begin{aligned}
f(t) &= f_{\mathrm{r}}(t) + f_{\mathrm{b}}(t) + f_{\mathrm{c}}(t), \quad f_{\mathrm{c}}(0)=0,\\
g(t) &= f_{\mathrm{r}}(0) + g_{\mathrm{c}}(t),\quad g_{\mathrm{c}}(0)=0.
\end{aligned}
\right.
\end{equation*}
Consequently, the difference of the two solutions is
\begin{equation*}
f(t)-g(t)=f_{\mathrm{r}}(t)-f_{\mathrm{r}}(0)+f_{\mathrm{b}}(t)+f_{\mathrm{c}}(t)-g_{\mathrm{c}}(t).
\end{equation*}
By applying Lemma \ref{lemma:Hs,bounds on fb}, Lemma \ref{lemma:bounds on fr}, and Proposition \ref{lemma:perturbation}, we have
\begin{equation}\label{equ:fb,fr,fc,proof}
\left\{
\begin{aligned}
\sup_{t \in [0, T_{*}]} \n{e^{\lra{v}^{2}} f_{\mathrm{b}}(t)}_{L_{v}^{2} H_{x}^{s_{0}}} &\lesssim M^{s_{0}-s}\lesssim \frac{1}{\ln M}, \\
\n{f_{\mathrm{r}}(0)}_{L_{v}^{2} H_{x}^{s_{0}}} &\gtrsim 1,\\
\n{e^{\lra{v}^{2}} f_{\mathrm{r}}(T_{*})}_{L_{v}^{2} H_{x}^{s_{0}}} &\lesssim \frac{1}{\ln M}, \\
\sup_{t \in [0, T_{*}]} \n{e^{(1-\kappa t) \lra{v}^{2}} f_{\mathrm{c}}(t)}_{L_{v}^{2} H_{x}^{s_{0}}}  \lesssim& M^{-\frac{\delta}{2}}\lesssim \frac{1}{\ln M}, \\
\sup_{t \in [0, T_{*}]} \n{e^{(1-\kappa t) \lra{v}^{2}} g_{\mathrm{c}}(t)}_{L_{v}^{2} H_{x}^{s_{0}}}  \lesssim& M^{-\frac{\delta}{2}}\lesssim \frac{1}{\ln M},
\end{aligned}
\right.
\end{equation}
where the final time $T_{*}=M^{s-1}(\ln\ln M)$.

For the initial time $t=0$, applying the first estimate in \eqref{equ:fb,fr,fc,proof} at $t=0$, we obtain
\begin{equation*}
\n{e^{\lra{v}^{2}} \lrs{f(0) - g(0)}}_{L_{v}^{2} H_{x}^{s_{0}}}
= \n{e^{\lra{v}^{2}} f_{\mathrm{b}}(0)}_{L_{v}^{2} H_{x}^{s_{0}}} \lesssim \frac{1}{\ln M},
\end{equation*}
which shows that the two initial data can be made arbitrarily close in the weighted $L_{v}^{2} H_{x}^{s_{0}}$ norm.

We now analyze the difference between the solutions at the prescribed final time $T_{*}$.
By the parameter choice $\kappa T_{*} = 1$, we have
\begin{align}\label{equ:kappa,T,1}
e^{(1-\kappa T_{*})\lra{v}^{2}}=1.
\end{align}
Furthermore, the difference of the two solutions at the final time gives that
\begin{equation*}
f(T_{*}) - g(T_{*}) = f_{\mathrm{r}}(T_{*}) - f_{\mathrm{r}}(0) + f_{\mathrm{b}}(T_{*}) + f_{\mathrm{c}}(T_{*}) - g_{\mathrm{c}}(T_{*}).
\end{equation*}
By taking the $L_{v}^{2}H_{x}^{s_{0}}$ norm, using \eqref{equ:kappa,T,1} at $t=T_{*}$, and applying the triangle inequality, we deduce
\begin{align}\label{equ:proof,T,difference}
&\n{f(T_{*}) - g(T_{*})}_{L_{v}^{2} H_{x}^{s_{0}}}\\
=& \n{e^{(1-\kappa T_{*})\lra{v}^{2}} \lrs{f(T_{*}) - g(T_{*})}}_{L_{v}^{2} H_{x}^{s_{0}}} \notag\\
\ge& \n{e^{(1-\kappa T_{*})\lra{v}^{2}} f_{\mathrm{r}}(0)}_{L_{v}^{2} H_{x}^{s_{0}}}
- \n{e^{(1-\kappa T_{*})\lra{v}^{2}} f_{\mathrm{r}}(T_{*})}_{L_{v}^{2} H_{x}^{s_{0}}}
- \n{e^{(1-\kappa T_{*})\lra{v}^{2}} f_{\mathrm{b}}(T_{*})}_{L_{v}^{2} H_{x}^{s_{0}}} \notag\\
&\quad - \n{e^{(1-\kappa T_{*})\lra{v}^{2}} f_{\mathrm{c}}(T_{*})}_{L_{v}^{2} H_{x}^{s_{0}}}
- \n{e^{(1-\kappa T_{*})\lra{v}^{2}} g_{\mathrm{c}}(T_{*})}_{L_{v}^{2} H_{x}^{s_{0}}}.\notag
\end{align}
Substitute the bounds from \eqref{equ:fb,fr,fc,proof} into the above inequality \eqref{equ:proof,T,difference}, we arrive at the lower bound
\begin{align*}
\n{f(T_{*}) - g(T_{*})}_{L_{v}^{2} H_{x}^{s_{0}}}
\gtrsim& \n{f_{\mathrm{r}}(0)}_{L_{v}^{2} H_{x}^{s_{0}}} - \frac{1}{\ln M}
\gtrsim 1.
\end{align*}
This establishes the desired separation at the finial time $t=T_{*}$, thereby completing the proof of the theorem.
\end{proof}

\section{Low-regularity Global Stability}\label{sec:Low-regularity Global Stability}
In this section, we consider the global stability for small initial data. Our analysis relies on dissipative energy estimates and dispersive effects. In Section \ref{section:Dissipation Energy Estimates}, we review the standard dissipative energy estimates. In Section \ref{section:Dispersive-type Bilinear Estimates}, we establish dispersive-type bilinear estimates, which are the key estimates for lowering the regularity to the sharp threshold $s>1$. Finally, in Section \ref{section:Proof of Global Stability}, we complete the proof of global stability.
\subsection{Dissipation Energy Estimates}\label{section:Dissipation Energy Estimates}
In this section, we introduce the standard dissipation energy estimates for the perturbation equation
\begin{align}\label{equ:Boltzmann,g,disspation}
\pa_{t}g+v\cdot \nabla_{x}g+Lg=\Gamma(g,g).
\end{align}
It is well-known that the linearized operator $L$ is coercive only on the microscopic component, and hence the dissipative energy estimates should be carried out by treating the macroscopic and microscopic parts separately.
We recall the macroscopic and microscopic decomposition of $g$ defined by
\begin{align*}
\mathbf{P}_{0} g=:&\left\{a+\mathbf{b} \cdot v+c\frac{|v|^2-3}{\sqrt{6}}\right\} \sqrt{\mu},\quad  \mathbf{P}_{1} g=:g-\mathbf{P}_{0} g.
\end{align*}
where $\mathbf{P}_{0}g$ and $\mathbf{P}_{1}g$ denote the macroscopic and microscopic parts, respectively. From this decomposition, it is straightforward to derive the projection estimate
\begin{align}\label{equ:projection estimate,P0}
\n{\lra{v}^{k}\mathbf{P}_{0}g}_{L_{v}^{2}}\lesssim \n{\mathbf{P}_{0}g}_{L_{v}^{2}}.
\end{align}
Furthermore, we denote by $P_{\leq 1}^{x}$ and $P_{>1}^{x}$ the Littlewood-Paley projections in the
$x$-variable, which correspond to low and high frequencies, respectively.

With these preparations, we now give the dissipation energy estimates in the following lemma.
\begin{lemma}[Dissipation energy estimates]\label{lemma:dissipation energy estimate}
Let $g$ satisfy \eqref{equ:Boltzmann,g,disspation} and
assume that $\mathbf{P}_{0}G=0$. Then we have
\begin{align}\label{equ:dissipation energy estimate,g}
\mathcal{E}(g)
\lesssim& \n{g(0)}_{L_{v}^{2}H_{x}^{s}}+\mathcal{N}(g).
\end{align}
where the energy $\mathcal{E}(g)$ and $\mathcal{N}(g)$ is given by
\begin{align}
\mathcal{E}(g):=&\n{g}_{L_{t}^{\infty}L_{v}^{2}H_{x}^{s}}+
\bn{\lra{v}^{\frac{1}{2}}\mathbf{P}_{1}g}_{L_{t}^{2}L_{v}^{2}H_{x}^{s}}+
\bbn{\frac{|\nabla_{x}|}{\lra{\nabla_{x}}}\mathbf{P}_{0}g}_{L_{t}^{2}L_{v}^{2}H_{x}^{s}},\label{equ:E,g,energy}\\
\mathcal{N}(g):=&\n{\lra{v}^{-\frac{1}{2}}(\pa_{t}+v\cdot \nabla_{x}+L)g}_{L_{t}^{2}L_{v}^{2}H_{x}^{s}}=\n{\lra{v}^{-\frac{1}{2}}\Gamma(g,g)}_{L_{t}^{2}L_{v}^{2}H_{x}^{s}}.
\label{equ:N,g,energy}
\end{align}
In particular, we have the $L_{t}^{2}$-type estimate for the high-frequency term and the $L_{x}^{6}$ estimate for the macroscopic part
\begin{align}
\n{\lra{v}^{\frac{1}{2}}P_{>1}^{x}g}_{L_{t}^{2}L_{v}^{2}H_{x}^{s}}\lesssim & \mathcal{E}(g),\label{equ:dissipation estimates,high frequency}\\
\n{\lra{v}^{\frac{1}{2}} \lra{\nabla_{x}}^{s}\mathbf{P}_{0}g}_{L_{t}^{2}L_{v}^{2}L_{x}^{6}}\lesssim& \mathcal{E}(g).\label{equ:dissipation estimates,low frequency,L6}
\end{align}
\end{lemma}
\begin{proof}
The dissipation energy estimates \eqref{equ:dissipation energy estimate,g} have been studied by many authors in various setting, such as \cite{Duan08,Duan11,DLS25,Guo03,Guo03vlasov,Guo04,Guo10,LYY04}. For completeness, we provide a proof following \cite[Chapter 4.1]{Duan11}. Applying the Fourier transform in the $x$-variable to
\eqref{equ:Boltzmann,g,disspation}, we obtain
\begin{align}\label{equ:Boltzmann,g,disspation,fourier}
\pa_{t}\widehat{g}+iv\cdot k \widehat{g}+L\widehat{g}=\widehat{\Gamma(g,g)}.
\end{align}

Testing \eqref{equ:Boltzmann,g,disspation,fourier} against $\ol{\widehat{g}}$,
integrating over the velocity variable, and using
the coercivity estimate $\lra{L\widehat{g},\widehat{g}}_{v}\gtrsim
\bn{\lra{v}^{\frac{1}{2}}\mathbf{P}_{1}\widehat{g}}_{L_{v}^{2}}^{2}$, we deduce
\begin{align}\label{equ:micro,energy,g,hat}
\frac{d}{dt}\n{\widehat{g}}_{L_{v}^{2}}^{2}+
\bn{\lra{v}^{\frac{1}{2}}\mathbf{P}_{1}\widehat{g}}_{L_{v}^{2}}^{2}
\lesssim \bn{\lra{v}^{-\frac{1}{2}}\widehat{\Gamma(g,g)}}_{L_{v}^{2}}.
\end{align}

This estimate \eqref{equ:micro,energy,g,hat} captures the dissipation of the microscopic part. Since the
linearized operator $L$ is not coercive on the macroscopic part
$\mathbf{P}_{0}\widehat{g}$, the corresponding dissipation must be recovered
separately. For this purpose, we invoke \cite[Lemma 4.1]{Duan11}, which provides a functional $\mathcal{E}_{\mathrm{int}}(\widehat{g})$ and a constant $\la>0$ such that
\begin{align}
\frac{d}{dt}\operatorname{Re}\mathcal{E}_{\mathrm{int}}(\widehat{g})+\la
\frac{|k|^{2}}{1+|k|^{2}}\bn{\mathbf{P}_{0}\widehat{g}}_{L_{v}^{2}}^{2}\lesssim &
\bn{\lra{v}^{\frac{1}{2}}\mathbf{P}_{1}\widehat{g}}_{L_{v}^{2}}^{2}+
\bn{\lra{v}^{-\frac{1}{2}}\widehat{\Gamma(g,g)}}_{L_{v}^{2}},\label{equ:macro,energy,Eint,hat}\\
|\mathcal{E}_{\mathrm{int}}(\widehat{g})|\lesssim& \n{\widehat{g}}_{L_{v}^{2}}^{2}
\label{equ:Eint,g,hat}.
\end{align}

Combining \eqref{equ:micro,energy,g,hat} and \eqref{equ:macro,energy,Eint,hat}, and
choosing $\ve>0$ sufficiently small, we derive
\begin{align*}
\frac{d}{dt}\n{\widehat{g}}_{L_{v}^{2}}^{2}+\ve\frac{d}{dt}\operatorname{Re}\mathcal{E}_{\mathrm{int}}(\widehat{g})
+\ve \la
\frac{|k|^{2}}{1+|k|^{2}}\bn{\mathbf{P}_{0}\widehat{g}}_{L_{v}^{2}}^{2}+\frac{1}{2}
\bn{\lra{v}^{\frac{1}{2}}\mathbf{P}_{1}\widehat{g}}_{L_{v}^{2}}^{2}\lesssim
\bn{\lra{v}^{-\frac{1}{2}}\widehat{\Gamma(g,g)}}_{L_{v}^{2}},
\end{align*}
Multiplying this inequality by $\lra{k}^{2s}$, integrating over $k$ and in time, and using \eqref{equ:Eint,g,hat} to absorb $\mathcal{E}_{\mathrm{int}}(\widehat{g})$, we then deduce the desired estimate \eqref{equ:dissipation energy estimate,g}.

For \eqref{equ:dissipation estimates,high frequency}, it follows from macroscopic-microscopic decomposition
\begin{align*}
\n{\lra{v}^{\frac{1}{2}}P_{>1}^{x}g}_{L_{t}^{2}L_{v}^{2}H_{x}^{s}}\lesssim & \n{\lra{v}^{\frac{1}{2}}\mathbf{P}_{1}P_{>1}^{x}g}_{L_{t}^{2}L_{v}^{2}H_{x}^{s}}+
\n{\lra{v}^{\frac{1}{2}}\mathbf{P}_{0} P_{>1}^{x}g}_{L_{t}^{2}L_{v}^{2}H_{x}^{s}}\lesssim \mathcal{E}(g),
\end{align*}
where in the last inequality we have used the projection estimate \eqref{equ:projection estimate,P0}.

For \eqref{equ:dissipation estimates,low frequency,L6},
by Sobolev inequality that $\dot{H}_{x}^{1}\hookrightarrow L_{x}^{6}$ and the projection estimate \eqref{equ:projection estimate,P0}, we have
\begin{align*}
\n{\lra{v}^{\frac{1}{2}}\lra{\nabla_{x}}^{s}\mathbf{P}_{0}g}_{L_{t}^{2}L_{v}^{2}L_{x}^{6}}\lesssim& \n{ \nabla_{x}\lra{\nabla_{x}}^{s}\mathbf{P}_{0}g}_{L_{t}^{2}L_{v}^{2}L_{x}^{2}}\lesssim
\bbn{\frac{|\nabla_{x}|}{\lra{\nabla_{x}}}\mathbf{P}_{0}g}_{L_{t}^{2}L_{v}^{2}H_{x}^{s}}\lesssim \mathcal{E}(g).
\end{align*}
Therefore, we have complete the proof of Lemma \ref{lemma:dissipation energy estimate}.
\end{proof}
We remark that the $L_{x}^{6}$ estimate \eqref{equ:dissipation estimates,low frequency,L6}
for the macroscopic part, introduced in \cite[Lemma 5.3]{Guo04}, is needed to close the low-frequency estimates
in Lemma \ref{lemma:dispersive,bilinear estimates}.

\subsection{Dispersive-type Bilinear Estimates}\label{section:Dispersive-type Bilinear Estimates}
Starting from the perturbation equation \eqref{equ:Boltzmann,g,disspation} with $L=\nu-K$, we rewrite it in the Duhamel form
\begin{align}
g(t)=S_{\nu}(t)g(0)+\int_{0}^{t}S_{\nu}(t-s)Kg(s)ds+
\int_{0}^{t}S_{\nu}(t-s)\Gamma(g,g)(s)ds,
\end{align}
where $S_{\nu}(t):=e^{-t\nu}S(t)$ is the damped transport operator and $S(t)=:e^{-tv\cdot \nabla_{x}}$ is the free transport operator.
We proceed to establish dispersive-type bilinear estimates, which serve as a foundation for the subsequent global stability analysis.
\begin{lemma}[Dispersive-type bilinear estimates]
Let $s>1$. There holds that
\begin{align}\label{equ:Gamma,f,g,s>1}
\n{\lra{v}^{-\frac{1}{2}}\Gamma^{\pm}(S_{\nu}(t)g_{1}(0),S_{\nu}(t)g_{2}(0))}_{L_{t}^{2}L_{v}^{2}H_{x}^{s}}\lesssim & \n{g_{1}(0)}_{L_{v}^{2}H_{x}^{s}}\n{g_{2}(0)}_{L_{v}^{2}H_{x}^{s}},\\
\bbn{\lra{v}^{-\frac{1}{2}}\Gamma^{\pm}\lrs{\int_{0}^{t}S_{\nu}(t-t_{1})g_{1}(t_{1})
dt_{1},S_{\nu}(t)g_{2}(0)}}_{L_{t}^{2}L_{v}^{2}H_{x}^{s}}
\lesssim &
\bn{\lra{v}^{-\frac{1}{2}}
g_{1}(t_{1})}_{L_{t_{1}}^{2}L_{v}^{2}H_{x}^{s}}\n{g_{2}(0)}_{L_{v}^{2}H_{x}^{s}},\label{equ:Gamma,f,t1,g,s>1}\\
\bbn{\lra{v}^{-\frac{1}{2}}\Gamma^{\pm}\lrs{S_{\nu}(t)g_{1}(0),\int_{0}^{t}S_{\nu}(t-t_{2})
g_{2}(t_{2})dt_{2}}}_{L_{t}^{2}L_{v}^{2}H_{x}^{s}}
\lesssim &
\n{g_{1}(0)}_{L_{v}^{2}H_{x}^{s}}
\bn{\lra{v}^{-\frac{1}{2}}
g_{2}(t_{2})}_{L_{t_{2}}^{2}L_{v}^{2}H_{x}^{s}},\label{equ:Gamma,f,g,t2,s>1}
\end{align}
\begin{align}\label{equ:Gamma,f,t1,g,t2,s>1}
&\bbn{\lra{v}^{-\frac{1}{2}}\Gamma^{\pm}\lrs{\int_{0}^{t}S_{\nu}(t-t_{1})g_{1}(t_{1})
dt_{1},\int_{0}^{t}S_{\nu}(t-t_{2})g_{2}(t_{2})dt_{2}}}_{L_{t}^{2}L_{v}^{2}H_{x}^{s}}\\
\lesssim &
\bn{\lra{v}^{-\frac{1}{2}}
g_{1}(t_{1})}_{L_{t_{1}}^{2}L_{v}^{2}H_{x}^{s}}\bn{\lra{v}^{-\frac{1}{2}}
g_{2}(t_{2})}_{L_{t_{2}}^{2}L_{v}^{2}H_{x}^{s}}.\notag
\end{align}
\end{lemma}
\begin{proof}
We first expand
\begin{align*}
\Gamma^{+}(F,G)
=&\int_{\R^{3}}\int_{\mathbb{S}^{2}}
\mu(v)|(u-v)\cdot \omega|F(t,x,v^{*})G(t,x,u^{*})d\omega du,\\
\Gamma^{-}(F,G)
=&\int_{\R^{3}}\int_{\mathbb{S}^{2}}
\mu(v)|(u-v)\cdot \omega|F(t,x,v)G(t,x,u)d\omega du.
\end{align*}
By Cauchy-Schwarz inequality, we obtain
\begin{align}
&\bn{\lra{v}^{-\frac{1}{2}}\Gamma^{+}(F,G)}_{L_{t}^{2}L_{v}^{2}H_{x}^{s}}\notag\\
\lesssim&\bn{\lra{v}^{-\frac{1}{2}}\mu(v)|(u-v)\cdot \omega|\lra{\nabla_{x}}^{s}\lrc{F(t,x,v^{*})G(t,x,u^{*})}}_{
L_{t}^{2}L_{x}^{2}L_{v}^{2}L_{u}^{1}L_{\omega}^{1}}\notag\\
\lesssim& \bn{\lra{v}^{-\frac{1}{2}}\mu(v)|u-v|^{\frac{1}{2}}}_{L_{v}^{\infty}L_{u}^{2}L_{\omega}^{2}}
\bn{|u-v|^{\frac{1}{2}}\lra{\nabla_{x}}^{s}\lrc{F(t,x,v^{*})
G(t,x,u^{*})}}_{L_{t}^{2}L_{x}^{2}L_{v}^{2}L_{u}^{2}L_{\omega}^{2}}.\notag
\end{align}
Noting that
\begin{align*}
\bn{\lra{v}^{-\frac{1}{2}}\mu(v)|u-v|^{\frac{1}{2}}}_{L_{v}^{\infty}L_{u}^{2}L_{\omega}^{2}}\lesssim 1,
\end{align*}
and using the collision transform $(u^{*},v^{*})\mapsto (u,v)$ together with $|u^{*}-v^{*}|=|u-v|$, we get
\begin{align}\label{equ:Gamma+,bilinear,space-time,FG}
\bn{\lra{v}^{-\frac{1}{2}}\Gamma^{+}(F,G)}_{L_{t}^{2}L_{v}^{2}H_{x}^{s}}\lesssim&\bn{|u-v|^{\frac{1}{2}}
\lra{\nabla_{x}}^{s}\lrc{F(t,x,v^{*})G(t,x,u^{*})}}_{L_{t}^{2}L_{x}^{2}L_{v}^{2}L_{u}^{2}L_{\omega}^{2}}\\
\lesssim&\bn{|u-v|^{\frac{1}{2}}\lra{\nabla_{x}}^{s}\lrc{F(t,x,v)G(t,x,u)}}_{L_{t}^{2}L_{x}^{2}L_{v}^{2}L_{u}^{2}}.\notag
\end{align}

Similarly, proceeding as in \eqref{equ:Gamma+,bilinear,space-time,FG}, we also have
\begin{align}\label{equ:Gamma-,bilinear,space-time,FG}
\bn{\lra{v}^{-\frac{1}{2}}\Gamma^{+}(F,G)}_{L_{t}^{2}L_{v}^{2}H_{x}^{s}} \lesssim&\bn{|u-v|^{\frac{1}{2}}\lra{\nabla_{x}}^{s}\lrc{F(t,x,v)G(t,x,u)}}_{L_{t}^{2}L_{x}^{2}L_{v}^{2}L_{u}^{2}}.
\end{align}

Now, we get into the analysis of estimate \eqref{equ:Gamma,f,g,s>1}. Using \eqref{equ:Gamma+,bilinear,space-time,FG} and \eqref{equ:Gamma-,bilinear,space-time,FG} with $F=S_{\nu}(t)g_{1}(0)$ and $G=S_{\nu}(t)g_{2}(0)$, we derive
\begin{align}
&\bn{\lra{v}^{-\frac{1}{2}}\Gamma^{\pm}(S_{\nu}(t)g_{1}(0),S_{\nu}(t)g_{2}(0))}_{L_{t}^{2}L_{v}^{2}H_{x}^{s}}
\label{equ:bilinear,space-time,f0g0}\\
\lesssim&\bn{|u-v|^{\frac{1}{2}}\lra{\nabla_{x}}^{s}\lrc{e^{-\nu(v)t}g_{1}(0,x-vt,v)
e^{-\nu(u)t}g_{2}(0,x-ut,u)}}_{L_{t}^{2}L_{x}^{2}L_{v}^{2}L_{u}^{2}}\notag\\
\lesssim&\bn{|u-v|^{\frac{1}{2}}\lra{\nabla_{x}}^{s}\lrc{g_{1}(0,x-vt,v)
g_{2}(0,x-ut,u)}}_{L_{t}^{2}L_{x}^{2}L_{v}^{2}L_{u}^{2}}.\notag
\end{align}
For simplicity, we set
\begin{align*}
H(t,x,u,v):=|u-v|^{\frac{1}{2}}\lra{\nabla_{x}}^{s}\lrc{g_{1}(0,x-vt,v)
g_{2}(0,x-ut,u)}.
\end{align*}

Therefore, the desired estimate \eqref{equ:Gamma,f,g,s>1} reduces to the $L^{2}$ estimate for $H(t,x,u,v)$. Such type estimates share a similar flavor as the space-time collapsing estimates for dispersive equations, as can be seen, for instance, in \cite[Theorem 1.3]{KM08}.
The idea is that the space-time Fourier transform converts the $L^{2}$ estimate into a problem of restricted-type surface integrals.
To proceed,
we take the space-time Fourier transform to obtain
\begin{equation}
\wt{H}(\tau,\xi,u,v)
=|u-v|^{\frac{1}{2}}\lra{\xi}^{s}
\int_{\R^{3}}
\wt{g_{1}}(0,\eta,v)\wt{g_{2}}(0,\xi-\eta,u)
\delta\lrs{\tau+u\cdot \xi+(u-v)\cdot \eta
}d\eta.
\end{equation}
By Cauchy-Schwarz inequality, we obtain
\begin{align}\label{equ:H,space-time fourier}
|\wt{H}(\tau,\xi,u,v)|^{2}
\lesssim& |I(\tau,\xi,u,v)|^{2}|J(\tau,\xi,u,v)|^{2},
\end{align}
where
\begin{align*}
I(\tau,\xi,u,v)=&\int_{\R^{3}} |u-v| \frac{\lra{\xi}^{2s}}{\lra{\eta}^{2s}\lra{\xi-\eta}^{2s}}
\delta\lrs{\tau+u\cdot \xi+(u-v)\cdot \eta
}d\eta,\\
J(\tau,\xi,u,v)=&\int_{\R^{3}}
|\lra{\eta}^{s}\wt{g_{1}}(0,\eta,v)|^{2}|\lra{\xi-\eta}^{s}\wt{g_{2}}(0,\xi-\eta,u)|^{2}
\delta\lrs{\tau+u\cdot \xi+(u-v)\cdot \eta
}d\eta.
\end{align*}

From \eqref{equ:H,space-time fourier}, it suffices to establish a uniform bound for
$I(\tau,\xi,u,v)$. By rotational symmetry, we begin with the following surface integral estimate with $s>1$
\begin{align}\label{equ:integration,estimate,delta,eta}
\int_{\R^{3}}\frac{1}{\lra{\eta}^{2s}}\delta(a+\eta \cdot w )d\eta
=&\int_{\R^{3}}\frac{1}{\lra{\eta}^{2s}}\delta(a+\eta_{1}|w| )d\eta\\
\lesssim &\frac{1}{|w|}\int_{\R^{2}}\lrs{\frac{1}{\sqrt{1+|\eta_{2}|^{2}+|\eta_{3}|^{2}}}}^{2s}d\eta_{2}d\eta_{3}\lesssim
\frac{1}{|w|}.\notag
\end{align}
Using \eqref{equ:integration,estimate,delta,eta}, we get
\begin{align}\label{equ:Linf,I,tau,xi}
I(\tau,\xi,u,v)\lesssim
\int_{\R^{3}} |u-v| \lrs{\frac{1}{\lra{\eta}^{2s}}+
\frac{1}{\lra{\xi-\eta}^{2s}}}
\delta\lrs{\tau+u\cdot \xi+(u-v)\cdot \eta
}d\eta\lesssim 1.
\end{align}
Putting together \eqref{equ:H,space-time fourier} and \eqref{equ:Linf,I,tau,xi}, by Plancherel's theorem, we arrive at
\begin{align}\label{equ:L2 space-time,St,f0,g0}
&\bn{|u-v|^{\frac{1}{2}}\lra{\nabla_{x}}^{s}\lrc{g_{1}(0,x-vt,v)
g_{2}(0,x-ut,u)}}_{L_{t}^{2}L_{x}^{2}L_{v}^{2}L_{u}^{2}}\\
\lesssim &
\n{\wt{H}(\tau,\xi,u,v)}_{L_{\tau}^{2}L_{\xi}^{2}L_{v}^{2}L_{u}^{2}}\notag\\
\lesssim&\n{I(\tau,\xi,u,v)}_{L_{\tau}^{\infty}L_{\xi}^{\infty}L_{v}^{\infty}L_{u}^{\infty}}
\n{J(\tau,\xi,u,v)}_{L_{\tau}^{2}L_{\xi}^{2}L_{v}^{2}L_{u}^{2}}\notag\\
\lesssim&\n{g_{1}(0)}_{L_{v}^{2}H_{x}^{s}}\n{g_{2}(0)}_{L_{v}^{2}H_{x}^{s}},\notag
\end{align}
which, together with \eqref{equ:bilinear,space-time,f0g0}, completes the proof of the desired estimate \eqref{equ:Gamma,f,g,s>1}.

Next, we proceed to the analysis of the estimates \eqref{equ:Gamma,f,t1,g,s>1}--\eqref{equ:Gamma,f,t1,g,t2,s>1}. In fact, it suffices to establish \eqref{equ:Gamma,f,t1,g,t2,s>1}, since the proof of the estimates \eqref{equ:Gamma,f,t1,g,s>1}--\eqref{equ:Gamma,f,g,t2,s>1} follows by similar arguments to those used for \eqref{equ:Gamma,f,g,s>1} and \eqref{equ:Gamma,f,t1,g,t2,s>1}.

Using again \eqref{equ:Gamma+,bilinear,space-time,FG} and \eqref{equ:Gamma-,bilinear,space-time,FG} with
$$F=\int_{0}^{t}S_{\nu}(t-t_{1})g_{1}(t_{1})
dt_{1},\quad G=\int_{0}^{t}S_{\nu}(t-t_{2})g_{2}(t_{2})dt_{2},$$
we obtain
\begin{align*}
I_{\eqref{equ:Gamma,f,t1,g,t2,s>1}}:=&\bbn{\lra{v}^{-\frac{1}{2}}\Gamma^{\pm}\lrs{\int_{0}^{t}S_{\nu}(t-t_{1})g_{1}(t_{1})
dt_{1},\int_{0}^{t}S_{\nu}(t-t_{2})g_{2}(t_{2})dt_{2}}}_{L_{t}^{2}L_{v}^{2}H_{x}^{s}}\\
\lesssim& \bbn{|u-v|^{\frac{1}{2}}\lra{\nabla_{x}}^{s}\lrc{\int_{0}^{t}S_{\nu}(t-t_{1})g_{1}(t_{1},x,v)
dt_{1}
\int_{0}^{t}S_{\nu}(t-t_{2})g_{2}(t_{2},x,u)dt_{2}}}_{L_{t}^{2}L_{x}^{2}L_{v}^{2}L_{u}^{2}}.
\end{align*}
Noting that $S_{\nu}(t-\tau)=e^{-\nu(t-\tau)}S(t-\tau)$ and
\begin{align*}
\bn{e^{-\nu(v)(t-t_{1})}
e^{-\nu(u)(t-t_{2})}1_{\lr{t\geq t_{1},t\geq t_{2}}}}_{L_{t_{1}}^{2}L_{t_{2}}^{2}}\lesssim
\lra{v}^{-\frac{1}{2}}\lra{u}^{-\frac{1}{2}},
\end{align*}
we use Cauchy-Schwarz inequality to get
\begin{align*}
I_{\eqref{equ:Gamma,f,t1,g,t2,s>1}}\lesssim&
\bn{|u-v|^{\frac{1}{2}}\lra{v}^{-\frac{1}{2}}\lra{u}^{-\frac{1}{2}}\lra{\nabla_{x}}^{s}\lrc{S(t-t_{1})g_{1}(t_{1},x,v)
S(t-t_{2})g_{2}(t_{2},x,u)}}_{L_{t}^{2}L_{x}^{2}L_{v}^{2}L_{u}^{2}L_{t_{1}}^{2}L_{t_{2}}^{2}}.
\end{align*}
Then, by Minkowski inequality and the $L^{2}$ estimate \eqref{equ:L2 space-time,St,f0,g0}, we derive
\begin{align*}
I_{\eqref{equ:Gamma,f,t1,g,t2,s>1}}\lesssim& \bbn{\bn{\lra{v}^{-\frac{1}{2}}S(-t_{1})g_{1}(t_{1})}_{L_{v}^{2}H_{x}^{s}}
\bn{\lra{u}^{-\frac{1}{2}}S(-t_{2})g_{2}(t_{2})}_{L_{u}^{2}H_{x}^{s}}}_{L_{t_{1}}^{2}L_{t_{2}}^{2}}\\
\lesssim&\bn{\lra{v}^{-\frac{1}{2}}
g_{1}(t_{1})}_{L_{t_{1}}^{2}L_{v}^{2}H_{x}^{s}}\bn{\lra{v}^{-\frac{1}{2}}
g_{2}(t_{2})}_{L_{t_{2}}^{2}L_{v}^{2}H_{x}^{s}},
\end{align*}
which completes the proof of \eqref{equ:Gamma,f,t1,g,t2,s>1}.
\end{proof}

\subsection{Proof of Global Stability}\label{section:Proof of Global Stability}
We first establish the closure estimates for the nonlinear terms, and then provide the proof of Theorem \ref{theorem:global stability}.
\begin{lemma}\label{lemma:dispersive,bilinear estimates}
Let $g_{i}$ satisfy the Boltzmann equation \eqref{equ:Boltzmann,g,disspation}. Then we have
\begin{align}
\bn{\lra{v}^{-\frac{1}{2}}\Gamma(P_{>1}^{x}g_{1},P_{\leq 1}^{x}g_{2})}_{L_{t}^{2}L_{v}^{2}H_{x}^{s}}\lesssim &
\mathcal{E}(g_{1})\mathcal{E}(g_{2}),\label{equ:high,low,g1,g2,Gamma}\\
\bn{\lra{v}^{-\frac{1}{2}}\Gamma(P_{\leq 1}^{x}g_{1},P_{> 1}^{x}g_{2})}_{L_{t}^{2}L_{v}^{2}H_{x}^{s}}\lesssim &
\mathcal{E}(g_{1})\mathcal{E}(g_{2}),\label{equ:low,high,g1,g2,Gamma}\\
\bn{\lra{v}^{-\frac{1}{2}}\Gamma(P_{\leq 1}^{x}g_{1},P_{\leq 1}^{x}g_{2})}_{L_{t}^{2}L_{v}^{2}H_{x}^{s}}\lesssim &
\mathcal{E}(g_{1})\mathcal{E}(g_{2}),\label{equ:low,low,g1,g2,Gamma}\\
\bn{\lra{v}^{-\frac{1}{2}}\Gamma(P_{>1}^{x}g_{1},P_{>1}^{x}g_{2})}_{L_{t}^{2}L_{v}^{2}H_{x}^{s}}\lesssim & \prod_{i=1,2}\lrs{\n{g_{i}(0)}_{L_{v}^{2}H_{x}^{s}}+
\mathcal{E}(g_{i})+\mathcal{N}(g_{i})},\label{equ:high,high,g1,g2,Gamma}
\end{align}
where $\mathcal{E}(g)$ and $\mathcal{N}(g)$ is given by \eqref{equ:E,g,energy}--\eqref{equ:N,g,energy}.
In particular, we have
\begin{align}\label{equ:g1,g2,Gamma,bilinear,estimate}
\n{\lra{v}^{-\frac{1}{2}}\Gamma(g_{1},g_{2})}_{L_{t}^{2}L_{v}^{2}H_{x}^{s}}\lesssim
\prod_{i=1,2}\lrs{\n{g_{i}(0)}_{L_{v}^{2}H_{x}^{s}}+\mathcal{N}(g_{i})}.
\end{align}
\end{lemma}
\begin{proof}
First, we give a preliminary estimate. By the fractional Leibniz rule, we have
\begin{align}
&\bn{\lra{v}^{-\frac{1}{2}}\Gamma(G_{1},G_{2})}_{L_{t}^{2}L_{v}^{2}H_{x}^{s}}\\
\lesssim& \bn{\lra{v}^{-\frac{1}{2}}\Gamma\lrs{\n{\lra{\nabla_{x}}^{s}G_{1}}_{L_{x}^{p_{1}}},
\n{G_{2}}_{L_{x}^{p_{2}}}}}_{L_{t}^{2}L_{v}^{2}}+
\bn{\lra{v}^{-\frac{1}{2}}\Gamma\lrs{\n{G_{1}}_{L_{x}^{q_{1}}},
\n{\lra{\nabla_{x}}^{s}G_{2}}_{L_{x}^{q_{2}}}}}_{L_{t}^{2}L_{v}^{2}}.\notag
\end{align}
Then applying the collision estimate that
\begin{align*}
\n{\lra{v}^{-\frac{1}{2}}\Gamma(g_{1},g_{2})}_{L_{v}^{2}}
\lesssim \n{\lra{v}^{\frac{1}{2}}g_{1}}_{L_{v}^{2}}\n{g_{2}}_{L_{v}^{2}}+
\n{g_{1}}_{L_{v}^{2}}\n{\lra{v}^{\frac{1}{2}}g_{2}}_{L_{v}^{2}},
\end{align*}
we get
\begin{align}\label{equ:Gamma,G1G2,I1234}
\bn{\lra{v}^{-\frac{1}{2}}\Gamma(G_{1},G_{2})}_{L_{t}^{2}L_{v}^{2}H_{x}^{s}}\lesssim& I_{1}(G_{1},G_{2})+I_{2}(G_{1},G_{2})+I_{3}(G_{1},G_{2})+I_{4}(G_{1},G_{2}),
\end{align}
where the terms $I_{j}(G_{1},G_{2})$ are given by
\begin{align*}
I_{1}(G_{1},G_{2})=:&\bn{\lra{v}^{\frac{1}{2}}\lra{\nabla_{x}}^{s}G_{1}}_{L_{t}^{r_{1}}L_{v}^{2}L_{x}^{p_{1}}}
\bn{G_{2}}_{L_{t}^{r_{1}^{*}}L_{v}^{2}L_{x}^{p_{2}}}\\
I_{2}(G_{1},G_{2})=:&\bn{\lra{\nabla_{x}}^{s}G_{1}}_{L_{t}^{r_{2}}L_{v}^{2}L_{x}^{p_{1}}}
\bn{\lra{v}^{\frac{1}{2}}G_{2}}_{L_{t}^{r_{2}^{*}}L_{v}^{2}L_{x}^{p_{2}}}\\
I_{3}(G_{1},G_{2})=:&\bn{\lra{v}^{\frac{1}{2}}G_{1}}_{L_{t}^{r_{3}}L_{v}^{2}L_{x}^{q_{1}}}
\bn{\lra{\nabla_{x}}^{s}G_{2}}_{L_{t}^{r_{3}^{*}}L_{v}^{2}L_{x}^{q_{2}}}\\
I_{4}(G_{1},G_{2})=:&\bn{G_{1}}_{L_{t}^{r_{4}}L_{v}^{2}L_{x}^{q_{1}}}
\bn{\lra{v}^{\frac{1}{2}}\lra{\nabla_{x}}^{s}G_{2}}_{L_{t}^{r_{4}^{*}}L_{v}^{2}L_{x}^{q_{2}}}
\end{align*}
with the exponents satisfying
\begin{align*}
\frac{1}{2}=\frac{1}{r_{i}}+\frac{1}{r_{i}^{*}},\quad \frac{1}{2}=\frac{1}{p_{1}}+\frac{1}{p_{2}},\quad \frac{1}{2}=\frac{1}{q_{1}}+\frac{1}{q_{2}}.
\end{align*}

\vspace{1em}
\textbf{High-low frequency estimates \eqref{equ:high,low,g1,g2,Gamma}--\eqref{equ:low,high,g1,g2,Gamma}.}
We only need to prove \eqref{equ:high,low,g1,g2,Gamma}, as the estimate \eqref{equ:low,high,g1,g2,Gamma} can be obtained in the same way by symmetry.

For $I_{1}$, we take $r_{1}=2$, $r_{1}^{*}=\infty$, $p_{1}=2$, $p_{2}=\infty$. Using Bernstein inequality and the high-frequency energy estimate \eqref{equ:dissipation estimates,high frequency}, we obtain
\begin{align}\label{equ:I1,high-low,g1g2}
I_{1}(P_{>1}^{x}g_{1},P_{\leq 1}^{x}g_{2})\lesssim &
\n{\lra{v}^{\frac{1}{2}}\lra{\nabla_{x}}^{s}P_{>1}^{x}g_{1}}_{L_{t}^{2}L_{v}^{2}L_{x}^{2}}
\n{P_{\leq 1}^{x}g_{2}}_{L_{t}^{\infty}L_{v}^{2}L_{x}^{\infty}}\\
\lesssim& \n{\lra{v}^{\frac{1}{2}}\lra{\nabla_{x}}^{s}P_{>1}^{x}g_{1}}_{L_{t}^{2}L_{v}^{2}L_{x}^{2}}
\n{P_{\leq 1}^{x}g_{2}}_{L_{t}^{\infty}L_{v}^{2}L_{x}^{2}}\lesssim \mathcal{E}(g_{1})\mathcal{E}(g_{2}).\notag
\end{align}

For $I_{2}$, we apply the macroscopic-microscopic decomposition and use Bernstein inequality to get
\begin{align}\label{equ:I2,high-low,g1g2}
I_{2}(P_{>1}^{x}g_{1},P_{\leq 1}^{x}g_{2})
\lesssim & I_{2}(P_{>1}^{x}g_{1},\mathbf{P}_{0}P_{\leq 1}^{x}g_{2})+I_{2}(P_{>1}^{x}g_{1},\mathbf{P}_{1}P_{\leq 1}^{x}g_{2})\\
\lesssim& \n{\lra{\nabla_{x}}^{s}P_{>1}^{x}g_{1}}_{L_{t}^{2}L_{v}^{2}L_{x}^{2}}
\n{\lra{v}^{\frac{1}{2}}\mathbf{P}_{0}P_{\leq 1}^{x}g_{2}}_{L_{t}^{\infty}L_{v}^{2}L_{x}^{\infty}}\notag
\\
&+\n{\lra{\nabla_{x}}^{s}P_{>1}^{x}g_{1}}_{L_{t}^{\infty}L_{v}^{2}L_{x}^{2}}
\n{\lra{v}^{\frac{1}{2}}\mathbf{P}_{1}P_{\leq 1}^{x}g_{2}}_{L_{t}^{2}L_{v}^{2}L_{x}^{\infty}}\notag\\
\lesssim& \mathcal{E}(g_{1})\mathcal{E}(g_{2}).\notag
\end{align}

For $I_{3}$, repeating the argument used for $I_{1}$ in \eqref{equ:I1,high-low,g1g2}, we get
\begin{align}\label{equ:I3,high-low,g1g2}
I_{3}(P_{>1}^{x}g_{1},P_{\leq 1}^{x}g_{2})\lesssim &
\n{\lra{v}^{\frac{1}{2}}P_{>1}^{x}g_{1}}_{L_{t}^{2}L_{v}^{2}L_{x}^{2}}
\n{\lra{\nabla_{x}}^{s}P_{\leq 1}^{x}g_{2}}_{L_{t}^{\infty}L_{v}^{2}L_{x}^{\infty}}\\
\lesssim& \n{\lra{v}^{\frac{1}{2}}P_{>1}^{x}g_{1}}_{L_{t}^{2}L_{v}^{2}L_{x}^{2}}
\n{\lra{\nabla_{x}}^{s}P_{\leq 1}^{x}g_{2}}_{L_{t}^{\infty}L_{v}^{2}L_{x}^{2}}\lesssim \mathcal{E}(g_{1})\mathcal{E}(g_{2}).\notag
\end{align}

For $I_{4}$, proceeding analogously to the estimate \eqref{equ:I2,high-low,g1g2} for $I_{2}$, we derive
\begin{align}\label{equ:I4,high-low,g1g2}
I_{4}(P_{>1}^{x}g_{1},P_{\leq 1}^{x}g_{2})
\lesssim & I_{4}(P_{>1}^{x}g_{1},\mathbf{P}_{0}P_{\leq 1}^{x}g_{2})+I_{4}(P_{>1}^{x}g_{1},\mathbf{P}_{1}P_{\leq 1}^{x}g_{2})\\
\lesssim& \n{P_{>1}^{x}g_{1}}_{L_{t}^{2}L_{v}^{2}L_{x}^{2}}
\n{\lra{v}^{\frac{1}{2}}\lra{\nabla_{x}}^{s}\mathbf{P}_{0}P_{\leq 1}^{x}g_{2}}_{L_{t}^{\infty}L_{v}^{2}L_{x}^{\infty}}\notag
\\
&+\n{P_{>1}^{x}g_{1}}_{L_{t}^{\infty}L_{v}^{2}L_{x}^{2}}
\n{\lra{v}^{\frac{1}{2}}\lra{\nabla_{x}}^{s}\mathbf{P}_{1}P_{\leq 1}^{x}g_{2}}_{L_{t}^{2}L_{v}^{2}L_{x}^{\infty}}\notag\\
\lesssim& \mathcal{E}(g_{1})\mathcal{E}(g_{2}).\notag
\end{align}

Putting together \eqref{equ:I1,high-low,g1g2}--\eqref{equ:I4,high-low,g1g2} for $I_{1}$--$I_{4}$, we complete the proof of the high-low frequency estimate \eqref{equ:high,low,g1,g2,Gamma}.

\vspace{1em}
\textbf{Low-low frequency estimate \eqref{equ:low,low,g1,g2,Gamma}.}

By macroscopic-microscopic decomposition, we have
\begin{align*}
\Gamma(P_{\leq 1}^{x}g_{1},P_{\leq 1}^{x}g_{2})
=\sum_{i=0,1}\sum_{j=0,1}
\Gamma(\mathbf{P}_{i}P_{\leq 1}^{x}g_{1},\mathbf{P}_{j}P_{\leq 1}^{x}g_{2}).
\end{align*}
We notice that the microscopic part admits dissipation energy estimates
\begin{align}
\n{\mathbf{P}_{1}g}_{L_{t}^{\infty}L_{v}^{2}H_{x}^{s}}+
\bn{\lra{v}^{\frac{1}{2}}\mathbf{P}_{1}g}_{L_{t}^{2}L_{v}^{2}H_{x}^{s}}\lesssim \mathcal{E}(g),
\end{align}
which are stronger than those available for the high-frequency term $P_{\geq 1}^{x}g$.
Consequently, the three terms involving $\mathbf{P}_{1}$, namely
\begin{align}
\Gamma(\mathbf{P}_{0}P_{\leq 1}^{x}g_{1},\mathbf{P}_{1}P_{\leq 1}^{x}g_{2}),\quad
\Gamma(\mathbf{P}_{1}P_{\leq 1}^{x}g_{1},\mathbf{P}_{0}P_{\leq 1}^{x}g_{2}),\quad
\Gamma(\mathbf{P}_{1}P_{\leq 1}^{x}g_{1},\mathbf{P}_{1}P_{\leq 1}^{x}g_{2}),
\end{align}
can be handled in the same way as in the high-low frequency estimate \eqref{equ:high,low,g1,g2,Gamma}.

It remains to treat the low-frequency macroscopic term $\Gamma(\mathbf{P}_{0}P_{\leq 1}^{x}g_{1},\mathbf{P}_{0}P_{\leq 1}^{x}g_{2})$.
Following the argument used for estimate \eqref{equ:Gamma,G1G2,I1234} with $p_{1}=q_{1}=6$ and $p_{2}=q_{2}=3$, we obtain
\begin{align*}
&\bn{\lra{v}^{-\frac{1}{2}}\Gamma(\mathbf{P}_{0}P_{\leq 1}^{x}g_{1},\mathbf{P}_{0}P_{\leq 1}^{x}g_{2})}_{L_{t}^{2}L_{v}^{2}H_{x}^{s}}\\
\lesssim& \n{\lra{v}^{\frac{1}{2}}\lra{\nabla_{x}}^{s}\mathbf{P}_{0}P_{\leq 1}^{x}g_{1}}_{L_{t}^{2}L_{v}^{2}L_{x}^{6}}
\n{\lra{v}^{\frac{1}{2}}\lra{\nabla_{x}}^{s}\mathbf{P}_{0}P_{\leq 1}^{x}g_{2}}_{L_{t}^{\infty}L_{v}^{2}L_{x}^{3}}
\end{align*}
Then, by Bernstein inequality, the projection estimate, and the $L_{x}^{6}$ estimate \eqref{equ:dissipation estimates,low frequency,L6}, we arrive at
\begin{align*}
&\bn{\lra{v}^{-\frac{1}{2}}\Gamma(\mathbf{P}_{0}P_{\leq 1}^{x}g_{1},\mathbf{P}_{0}P_{\leq 1}^{x}g_{2})}_{L_{t}^{2}L_{v}^{2}H_{x}^{s}}\\
\lesssim& \n{\lra{\nabla_{x}}^{s}\mathbf{P}_{0}P_{\leq 1}^{x}g_{1}}_{L_{t}^{2}L_{v}^{2}L_{x}^{6}}
\n{\lra{\nabla_{x}}^{s}\mathbf{P}_{0}P_{\leq 1}^{x}g_{1}}_{L_{t}^{\infty}L_{v}^{2}L_{x}^{2}}\lesssim \mathcal{E}(g_{1})\mathcal{E}(g_{2}),
\end{align*}
which completes the proof of \eqref{equ:low,low,g1,g2,Gamma}.

\vspace{1em}
\textbf{High-high frequency estimate \eqref{equ:high,high,g1,g2,Gamma}.}

We recall that
\begin{align*}
g_{i}(t)=S_{\nu}(t)g_{i}(0)+\int_{0}^{t}S_{\nu}(t-s)Kg_{i}(s)ds+
\int_{0}^{t}S_{\nu}(t-s)\Gamma(g_{i},g_{i})(s)ds,\quad i=1,2.
\end{align*}
Using dispersive-type bilinear estimates \eqref{equ:Gamma,f,g,s>1}--\eqref{equ:Gamma,f,t1,g,t2,s>1}, we get
\begin{align*}
&\bn{\lra{v}^{-\frac{1}{2}}\Gamma(P_{>1}^{x}g_{1},P_{>1}^{x}g_{2})}_{L_{t}^{2}L_{v}^{2}H_{x}^{s}}\\
\lesssim& \prod_{i=1,2}\lrs{\n{g_{i}(0)}_{L_{v}^{2}H_{x}^{s}}+\bn{\lra{v}^{-\frac{1}{2}}
KP_{>1}^{x}g_{i}(t_{1})}_{L_{t_{i}}^{2}L_{v}^{2}H_{x}^{s}}+
\n{\lra{v}^{-\frac{1}{2}}\Gamma(g_{i},g_{i})}_{L_{t_{i}}^{2}L_{v}^{2}H_{x}^{s}}}.
\end{align*}
By the collision estimate that $\n{Kf}_{L_{v}^{2}}\lesssim \n{f}_{L_{v}^{2}}$ and the dissipation energy estimate \eqref{equ:dissipation estimates,high frequency} for the high frequency, we deduce
\begin{align*}
&\bn{\lra{v}^{-\frac{1}{2}}\Gamma(P_{>1}^{x}g_{1},P_{>1}^{x}g_{2})}_{L_{t}^{2}L_{v}^{2}H_{x}^{s}}\\
\lesssim& \prod_{i=1,2}\lrs{\n{g_{i}(0)}_{L_{v}^{2}H_{x}^{s}}+\bn{\lra{v}^{-\frac{1}{2}}
P_{>1}^{x}g_{i}(t_{i})}_{L_{t_{i}}^{2}L_{v}^{2}H_{x}^{s}}+
\n{\lra{v}^{-\frac{1}{2}}\Gamma(g_{i},g_{i})}_{L_{t_{i}}^{2}L_{v}^{2}H_{x}^{s}}}\\
\lesssim& \prod_{i=1,2}\lrs{\n{g_{i}(0)}_{L_{v}^{2}H_{x}^{s}}+
\mathcal{E}(g_{i})+\mathcal{N}(g_{i})}.
\end{align*}
Therefore, we complete the proof of \eqref{equ:high,high,g1,g2,Gamma}.

Finally, for \eqref{equ:g1,g2,Gamma,bilinear,estimate},
putting together estimates \eqref{equ:high,low,g1,g2,Gamma}--\eqref{equ:high,high,g1,g2,Gamma}, we use the dissipation energy estimate \eqref{equ:dissipation energy estimate,g} to derive
\begin{align*}
\n{\lra{v}^{-\frac{1}{2}}\Gamma(g_{1},g_{2})}_{L_{t}^{2}L_{v}^{2}H_{x}^{s}}\lesssim&
\prod_{i=1,2}\lrs{\n{g_{i}(0)}_{L_{v}^{2}H_{x}^{s}}+
\mathcal{E}(g_{i})+\mathcal{N}(g_{i})}\\
\lesssim&\prod_{i=1,2}\lrs{\n{g_{i}(0)}_{L_{v}^{2}H_{x}^{s}}+\mathcal{N}(g_{i})}.
\end{align*}
Hence, we complete the proof of Lemma \ref{lemma:dispersive,bilinear estimates}.
\end{proof}

\begin{proof}[\textbf{Proof of Theorem \ref{theorem:global stability}}]
Let $g$ satisfy the Boltzmann equation \eqref{equ:perturbation,Boltzmann,g}.
We introduce the $X$-norm
\begin{align*}
\n{g}_{X}:=\mathcal{E}(g)+\n{\lra{v}^{-\frac{1}{2}}(\pa_{t}+v\cdot \nabla_{x}+L)g}_{L_{t}^{2}L_{v}^{2}H_{x}^{s}}=
\mathcal{E}(g)+\mathcal{N}(g),
\end{align*}
and define the nonlinear map
\begin{align}\label{equ:Phig,SL,duhamel}
\Phi(g):=S_{L}(t)g(0)+\int_{0}^{t}S_{L}(t-s)\Gamma(g,g)(s)ds,\quad S_{L}(t)=e^{-t(v\cdot \nabla_{x}+L)},
\end{align}
which implies that
$(\pa_{t}+v\cdot \nabla_{x}+L)\Phi(g)=\Gamma(g,g)$.

Using dissipation energy estimate \eqref{equ:dissipation energy estimate,g} that
\begin{align*}
\mathcal{E}(g)\lesssim \n{g(0)}_{L_{v}^{2}H_{x}^{s}}+\mathcal{N}(g),
\end{align*}
and bilinear estimates \eqref{equ:g1,g2,Gamma,bilinear,estimate} for the nonlinear terms, we obtain
\begin{align}
\n{\Phi(g)}_{X}=&\mathcal{E}(\Phi(g))+
\n{\lra{v}^{-\frac{1}{2}}(\pa_{t}+v\cdot \nabla_{x}+L)\Phi(g)}_{L_{t}^{2}L_{v}^{2}H_{x}^{s}}\\
\lesssim& \n{g(0)}_{L_{v}^{2}H_{x}^{s}}+\n{\lra{v}^{-\frac{1}{2}}\Gamma(g,g)}_{L_{t}^{2}L_{v}^{2}H_{x}^{s}}\notag\\
\lesssim&\n{g(0)}_{L_{v}^{2}H_{x}^{s}}+\lrs{\n{g(0)}_{L_{v}^{2}H_{x}^{s}}+\n{g}_{X}}^{2}.\notag
\end{align}
and
\begin{align}
\n{\Phi(g_{1})-\Phi(g_{2})}_{X}
=&\mathcal{E}(\Phi(g_{1})-\Phi(g_{2}))+
\n{\lra{v}^{-\frac{1}{2}}(\pa_{t}+v\cdot \nabla_{x}+L)(\Phi(g_{1})-\Phi(g_{2}))}_{L_{t}^{2}L_{v}^{2}H_{x}^{s}}\notag\\
\lesssim& \bn{\lra{v}^{-\frac{1}{2}}\lrc{\Gamma(g_{1},g_{1})-\Gamma(g_{2},g_{2})}}_{L_{t}^{2}L_{v}^{2}H_{x}^{s}}\notag\\
\lesssim& \n{g_{1}-g_{2}}_{X}\lrs{\n{g(0)}_{L_{v}^{2}H_{x}^{s}}+\n{g_{1}}_{X}+\n{g_{2}}_{X}},\notag
\end{align}
These two estimates imply that $\Phi$ is a contraction map, provided that
$\n{g(0)}_{L_{v}^{2}H_{x}^{s}}$ is sufficiently small.
Hence, the global well-posedness follows from the contraction mapping principle.

Furthermore, in view of the relation $f=\mu^{\frac{1}{2}}g$, we conclude the global well-posedness for the perturbation equation \eqref{equ:perturbation,Boltzmann,f} for $f$. Consequently, the global stability estimate \eqref{equ:global,f,f0} follows, and the solution map is Lipschitz continuous.
\end{proof}

\noindent \textbf{Acknowledgements}
The authors would like to thank Professors Yan Guo, Justin Holmer, and Haitao Wang for helpful discussions.
S. Shen was supported in part by NSF of China under Grant 12501322.
Z. Zhang was supported in part by NSF of China under Grant 12288101.
\vspace{0.5em}

\noindent\textbf{Disclosure of AI assistance.}
This work continues the research line of \cite{CDP19,CH24well,CSZ23sharp,CSZ24well,CGSZ26} and was written by the authors.
ChatGPT 5.6 (OpenAI) was used to assist with the literature search, in which some nice references such as \cite{Duan08,Duan11} previously unfamiliar to the authors were found.
The authors take full responsibility for the accuracy, originality, and entire content of the manuscript.

\vspace{0.5em}
\noindent\textbf{Data Availability Statement}
Data sharing is not applicable to this article as no datasets were generated or analysed during the current study.

\vspace{0.5em}
\noindent\textbf{Conflict of Interest}
The authors declare that they have no conflict of interest.

\bibliographystyle{abbrv}
\bibliography{references}

\end{document}